\documentclass[10pt]{amsart}
\usepackage{amsmath,amsfonts,amssymb,amsthm}
\usepackage[alphabetic]{amsrefs}
\usepackage{hyperref}
\usepackage{graphicx,color}

\newtheorem{theorem}{Theorem}[section]
\newtheorem{lemma}[theorem]{Lemma}

\newtheorem{proposition}[theorem]{Proposition}

\newtheorem*{theorem*}{Theorem}
\newtheorem*{propertyA}{Property A}
\newtheorem*{corollaryA}{Corollary A}
\newtheorem*{corollaryB}{Corollary B}
\theoremstyle{definition}

\theoremstyle{remark}
\newtheorem*{remark}{Remark}

\usepackage{caption}
\begin{document}

\title[\ ]{The collapsibility of some CAT(0) simplicial complexes of dimension $3$}

\author[\ ]{
Ioana-Claudia Laz\u{a}r\\
'Politehnica' University of Timi\c{s}oara, Dept. of Mathematics,\\
Victoriei Square $2$, $300006$-Timi\c{s}oara, Romania\\
E-mail address: ioana.lazar@upt.ro}

\date{}

\begin{abstract}
We study the collapsibility of finite simplicial complexes of dimension $3$ endowed with a CAT(0)
metric. Our main result states that, under an additional hypothesis, finite simplicial $3$-complexes
endowed with a CAT(0) metric collapse
to a point through CAT(0) subspaces.

\hspace{0 mm} \textbf{2010 Mathematics Subject Classification}:
05C99, 05C75.

\hspace{0 mm} \textbf{Keywords}: simplicial $3$-complex, CAT(0) metric, strongly convex metric,
elementary collapse, geodesic segment. \\

\end{abstract}

\pagestyle{myheadings}

\markboth{}{}

  \vspace{-10pt}

\maketitle

\section{Introduction}

In this paper we find a sufficient condition for the collapsibility of a particular class of finite
simplicial complexes of dimension $3$. Namely, we show that the existence of a CAT(0) metric guarantees the collapsibility
of those complexes which satisfy a so called Property A. Roughly, Property A refers to preserving the strongly convex metric on a
subcomplex obtained by performing an elementary collapse on a finite CAT(0) $3$-complex.
Property A imposes restrictions only when deleting a $3$-simplex by starting at its free face. A similar restriction is not
encountered when deleting a $2$-simplex by starting at its free face.

The collapsibility of finite simplicial complexes was studied before. In \cite{white_1967} it is
shown that finite, strongly convex simplicial complexes of dimension $2$ are collapsible, whereas
in dimension $3$ such complexes collapse to a $2$-dimensional spine. It is the paper's object to
show that in dimension $3$ a stronger metric condition given by the CAT(0) metric, ensures, under additional assumptions,
collapsibility not only to a spine of dimension $2$, but even to a point.

Using discrete Morse theory (see \cite{forman_1998}), Crowley proved in $2008$, under a technical condition, that
nonpositively curved simplicial complexes of dimension $3$ or less endowed with the standard
piecewise Euclidean metric, collapse to a point (see \cite{crowley_2008}). She constructed a
CAT(0) triangulated disk by endowing it with the standard piecewise Euclidean metric and requiring
that each of
its interior vertices has degree at least $6$. The naturally associated standard piecewise Euclidean
metric on the disk became then CAT(0).

Adiprasito and Benedetti extended Crowley's result to all
dimensions (see \cite{benedetti_2012}, Theorem $3.2.1$). Namely,
they proved using discrete Morse theory that every complex that is
CAT(0) with a metric for which all vertex stars are convex, is
collapsible. It is important to note that, although the $3$-complexes in
our paper are also CAT(0) spaces,
they are no longer necessarily endowed with the standard piecewise Euclidean metric like the ones in
Crowley's and Adiprasito and Benedetti's papers. Still, they can also be collapsed to a point.

In \cite{baralic_2014} we show further, using again discrete Morse theory, that systolic simplicial complexes (see \cite{JS1})
are also collapsible. Moreover, we prove that both systolic and CAT(0) locally finite simplicial complexes possess an arborescent structure.
The collapsibility of systolic simplicial complexes is also proven by Chepoi and Osajda in \cite{chepoi_2009} (see Corollary $4.3$).

It is known that in dimension $2$ the CAT(0) metric guarantees the
collapsibility, through CAT(0) subspaces, of finite simplicial complexes, not necessarily endowed
with the standard piecewise Euclidean metric and whose interior vertices do not necessarily have
degree at least $6$ like the ones in
Crowley's  paper(see \cite{lazar_2010_8}, chapter $3.1$, page $35$). In this paper we
extend this result to dimension $3$. Namely, we show that, in certain circumstances, finite, CAT(0) simplicial $3$-complexes can
be collapsed to a point through subspaces which are, at each step of the retraction, endowed with a CAT(0) metric.
The result in dimension $3$ works only under an additional Property A given below.

\begin{propertyA}
Let $K$ be a finite CAT(0) simplicial $3$-complex and let $\sigma$ be a $3$-simplex of $K$ with a free $n$-face $\alpha, 1 \leq n \leq 2$. Let
$K' = K \setminus \{\sigma, \alpha \}$ be the subcomplex obtained by performing an elementary collapse on $K$. Let $p,q$ be two points of $K$ which do not belong to $\sigma$
such that the geodesic segment $[p,q]$ intersects the interior of $\sigma$.
Let $U$ be a small neighborhood of some vertex of $\sigma$ such that $\sigma$ is included in $U$. Let $U' = U \setminus \{ \sigma, \alpha \}$
Then in $U'$ there do not exist two geodesic segments $\gamma_{1}, \gamma_{2}$ of equal length joining $p$ to $q$ such that
$\gamma_{1}$ intersects one, while $\gamma_{2}$ intersects one or two of the three boundary edges of $\sigma$ which differ from any of the boundary edges of
$\alpha$ (if $\alpha$ is $2$-dimensional) or from $\alpha$ itself (if $\alpha$ is $1$-dimensional).

\end{propertyA}

Note that in general a finite CAT(0) $3$-complex can not be simplicially collapsed to a point because once performing the first elementary collapse on the complex,
the subcomplex we obtain does not inherit the strongly convex metric. This happens because the situation we exclude by imposing Property A on the complex, may in general occur.
Our proof relies on the definition of an elementary collapse. It uses basic properties of
CAT(0) spaces (see
\cite{bridson_1999}, \cite{burago_2001}, \cite{alex_1955}) and one of White's results
given in \cite{white_1967}.
Namely, because CAT(0) spaces have a strongly convex metric, finite, CAT(0)
$3$-complexes have, according to White, a $3$-simplex with a free face. One can therefore perform an
elementary collapse on such complex.
We show that the subcomplex obtained by performing an elementary collapse on a CAT(0)
$3$-complex enjoying Property A remains, at any step of the retraction, nonpositively curved. An important issue to
solve will be to find the new geodesic segments in the neighborhood of each point of the
subcomplex obtained by performing any step of the elementary collapse.

\textbf{Acknowledgements.} The author was partially supported by Project $19/6-020/961-120/14$ of Ministry of Science of the
Republic of Srpska.

\section{Preliminaries}

We present in this section the notions we shall work with and
the results we shall refer to.

Let $(X,d)$ be a metric space. Let $a,b \in \mathbf{R}$ such that
$[a,b]$ is a real interval. A \emph{geodesic path} joining $x \in
X$ to $y \in X$ is a path $c : [a,b] \to X$ such that $c(a) = x$,
$c(b) = y$ and $d(c(t), c(t')) = |t - t'|$ for all $t, t' \in
[a,b]$. The image $\alpha$ of $c$ is called a \emph{geodesic
segment} with endpoints $x$ and $y$. Since geodesic segments in
$\mathbf{R}$ are just closed intervals, this is a legitime abuse
of notation.

A \emph{geodesic metric space} $(X,d)$ is a metric space in which
every pair of points can be joined by a geodesic segment. We
denote any geodesic segment from a point $x$ to a point $y$ in
$X$, by $[x,y]$.

Given a path $c : [0,1] \rightarrow X$, its \emph{length} is defined by
\begin{center} $l(c) = \sup \{\sum_{i = 1}^{n}d(c(t_{i-1}), c(t_{i}))\}$,\end{center}
where the supremum is taken
over all possible subdivisions of $[0,1]$, $0 = t_{0} < t_{1} < ... < t_{n} = 1$.

Let $(X,d)$ be a geodesic metric space.
A \emph{geodesic triangle} in $X$ consists of three distinct points $x_{1}, x_{2}, x_{3}
\in X$, called \emph{vertices}, and a choice of three geodesic
segments joining them, called \emph{sides}.
Such a geodesic triangle is denoted by $\triangle = \triangle (x_{1}, x_{2}, x_{3})$. If a point $a
\in X$ lies
in the union of $[x_{1}, x_{2}], [x_{2}, x_{3}]$ and $[x_{3}, x_{1}]$, then we write $a \in
\triangle$. A triangle $\overline{\triangle} =
\overline{\triangle}(x_{1}, x_{2}, x_{3}) =
\triangle(\overline{x_{1}},\overline{x_{2}},\overline{x_{3}})$ in
$\mathbf{R}^{2}$ is called a \emph{comparison triangle} for
$\triangle$ if $d(x_{i}, x_{j}) =
d_{\mathbf{R}^{2}}(\overline{x_{i}},\overline{x_{j}})$, $i,j \in \{ 1,2,3 \}$. A point
$\overline{a} \in [\overline{x_{1}},\overline{x_{2}}]$ is called a
\emph{comparison point} for $a \in [x_{1}, x_{2}]$ if $d(x_{1},a) =
d_{\mathbf{R}^{2}}(\overline{x_{1}},\overline{a})$.
The interior angle of $\overline{\triangle}$ at $\overline{x_{1}}$
is called the \emph{comparison angle} between $x_{2}$ and $x_{3}$ at $x_{1}$.
A \emph{tetrahedron} in $X$
is the union of four geodesic triangles any two of which have
exactly one side in common.

Let $c,c',c''$ be three geodesic paths in $X$ issuing from the same point $x$. The
\emph{Aleksandrov angle} between
$c$ and $c'$ at $x$ is defined as
\begin{center}$\angle(c,c') = \lim \sup_{s,t \rightarrow 0}
\overline{\angle}_{x}(c(s),c'(t)) \in [0,\pi]$, $s,t \in [0,1]$,\end{center}
where $\overline{\angle}_{x}(c(s),c'(t))$ is the angle at the vertex
corresponding to $x$ in a comparison triangle in $\mathbb{R}^{2}$ for the geodesic triangle
$\triangle(x, c(s), c(t))$ in $X$.
The following inequality holds \begin{center}$\angle(c',c'') \leq \angle(c,c') + \angle(c,c'')$\end{center}
(for the proof see \cite{bridson_1999}, chapter I.$1$, page $10$). Alexandrov angles in
$\mathbf{R}^{2}$ are the usual Euclidean angles.

Let $\triangle = \triangle (p,q,r)$ be a
geodesic triangle in a convex metric space $(X,d)$ and let $\alpha, \beta, \gamma$ denote the
Alexandrov angles between the sides of $\triangle$. We define the
\emph{curvature of} $\triangle$ by $\omega(\triangle) = \alpha + \beta + \gamma - \pi$.
Any geodesic triangle in $X$ of curvature zero
is isometric to its comparison triangle in
$\mathbf{R}^{2}$ (for the proof see \cite{alex_1955}, chapter
V.$6$, page $218$).

Let $(X,d)$ be a metric space. We call $X$ a \emph{CAT(0) space} if it is a geodesic
space all of whose geodesic triangles satisfy the so called CAT(0) inequality. Namely,
for any geodesic triangle $\triangle \subset X$ and for any $x,y \in \triangle$,
\begin{center}$d(x,y) \leq d_{\mathbf{R}^{2}}(\overline{x}, \overline{y})$,\end{center}
where $\overline{x}, \overline{y} \in \overline{\triangle}$ are the corresponding comparison
points in the comparison triangle $\overline{\triangle}$ of $\triangle$ in $\mathbf{R}^{2}$.
We call $X$ \emph{nonpositively curved} if it is locally a CAT(0) space, i.e.
for every $x \in X$, there exists $r_{x} > 0$ such that the ball
$B(x, r_{x})$, endowed with the induced metric, is a CAT(0) space.

A \emph{subembedding} in
$\mathbf{R}^{2}$ of a $4-$tuple of points $(x_{1}, y_{1}, x_{2}, y_{2})$
in $X$ is a $4-$tuple of points $(\overline{x_{1}}, \overline{y_{1}},
\overline{x_{2}}, \overline{y_{2}})$ in $\mathbf{R}^{2}$ such that
$d_{\mathbf{R}^{2}}(\overline{x_{i}}, \overline{y_{j}}) = d(x_{i}, y_{j})$, $i,j \in \{1,2\}$,
$d(x_{1}, x_{2}) \leq d_{\mathbf{R}^{2}}(\overline{x_{1}}, \overline{x_{2}})$ and $d(y_{1},
y_{2}) \leq d_{\mathbf{R}^{2}}(\overline{y_{1}}, \overline{y_{2}})$.
We say $X$ \textit{satisfies the CAT(0)} $4-$\textit{point condition} if every $4$-tuple of points
$(x_{1}, y_{1}, x_{2}, y_{2})$ in $X$ has a subembedding in $\mathbf{R}^{2}$.

A metric space is a CAT(0) space if
and only if it is a geodesic space and if, for each of its geodesic triangles $\triangle$, the
Aleksandrov angle at any vertex of $\triangle$ is not greater than the corresponding angle in its
comparison triangle $\overline{\triangle}$ in $\mathbf{R}^{2}$ (for the proof see
\cite{bridson_1999}, chapter II.$1$, page $161$). Any complete, CAT(0) space satisfies the CAT(0)
$4$-point condition (for the proof see \cite{bridson_1999}, chapter II.$1$, page $164$). Any
complete,
simply connected, nonpositively curved space is a CAT(0) space (for
the proof see \cite{bridson_1999}, chapter II.$4$, page $194$).

Let $(X,d)$ be a CAT(0) space. The distance function $d : X \times X \rightarrow \mathbf{R}$ is
convex (for the proof see \cite{bridson_1999}, chapter II.$2$, page $176$) and strongly convex
(for
the proof see \cite{bridson_1999}, chapter II.$2$, page $160$). Any CAT(0) space is
contractible and hence simply connected (for the proof see \cite{bridson_1999}, chapter II.$2$,
page $161$).
The balls in $X$ are convex spaces (for the proof see \cite{bridson_1999},
chapter II.$1$, page $160$). For every $\varepsilon > 0$ there
exists $\delta = \delta(\varepsilon) > 0$
such that if $m$ is the midpoint of a geodesic segment $[x,y]
\subset X$ and if $\max \{ d(x,m'), d(y,m') \} \leq\frac{\textstyle 1}{\textstyle
2}d(x,y) + \delta$, then $d(m,m') < \varepsilon$ (for the proof see \cite{bridson_1999},
chapter II.$1$, page $160$). For $p,x,y \in X$, the geodesic segment $[x,y]$ is the union of the
geodesic segments
$[x,p]$ and $[p,y]$ if and only if $\angle_{p}(x,y) = \pi$ (see \cite{bridson_1999},
chapter II.$1$, page $163$).

We will make frequent use of Aleksandrov's Lemma
given below (for the proof see \cite{bridson_1999}, chapter I.$2$, page $25$).

\begin{lemma}\label{1.1.25}
Let $a,b,c,d$ be points in $\mathbb{R}^{2}$ such that $a$ and $c$ are in
different half-planes with respect to the line $bd$.
Consider a triangle $\triangle(a',b',c')$ in $\mathbb{R}^{2}$ such that $d_{\mathbf{R}^{2}}(a,b) =
d_{\mathbf{R}^{2}}(a',b')$, $d_{\mathbf{R}^{2}}(b,c) = d_{\mathbf{R}^{2}}(b',c')$,
$d_{\mathbf{R}^{2}}(a,d) +
d_{\mathbf{R}^{2}}(d,c) = d_{\mathbf{R}^{2}}(a',c')$ and let $d'$ be a point on the segment
$[a',c']$ such that $d_{\mathbf{R}^{2}}(a,d) = d_{\mathbf{R}^{2}}(a',d')$.

Then $\angle_{d}(a,b) + \angle_{d}(b,c) < \pi$ if and only if $d_{\mathbf{R}^{2}}(b',d') <
d_{\mathbf{R}^{2}}(b,d)$. In this case, one also has $\angle_{a'}(b',d') < \angle_{a}(b,d)$ and
$\angle_{c'}(b',d') < \angle_{c}(b,d)$.

Furthermore $\angle_{d}(a,b) + \angle_{d}(b,c) > \pi$ if and only if $d_{\mathbf{R}^{2}}(b',d') >
d_{\mathbf{R}^{2}}(b,d)$. In this case, one also has $\angle_{a'}(b',d') > \angle_{a}(b,d)$ and
$\angle_{c'}(b',d') > \angle_{c}(b,d)$.

Any one equality implies the others and occurs if and only if $\angle_{d}(a,b) + \angle_{d}(b,c) =
\pi$.
\end{lemma}

Let $K$ be a simplicial complex and let $\alpha$ be an $i$-simplex
of $K$. If $\beta$ is a $k$-dimensional face of $\alpha$ but not of
any other simplex in $K$, then we say there is an \emph{elementary
collapse} from $K$ to $K \setminus \{\alpha, \beta\}$. If $K = K_{0}
\supseteq K_{1} \supseteq ... \supseteq K_{n} = L$ are simplicial
complexes such that there is an elementary collapse from $K_{j-1}$
to $K_{j}$, $1 \leq j \leq n$, then we say that $K$
\emph{simplicially collapses} to $L$.

Let $K$ be a finite, connected simplicial complex endowed with the
standard piecewise Euclidean metric. We define the \emph{standard
piecewise Euclidean metric} on $|K|$ by taking the distance between
any two points $x,y$ in $|K|$ to be the infimum over all paths in
$|K|$ from $x$ to $y$. Each simplex of $K$ is isometric with a
regular Euclidean simplex of the same dimension with side lengths
equal $1$.

\section{Collapsing certain CAT(0) simplicial complexes of dimension $3$}

In this section we prove that finite, CAT(0)
simplicial $3$-complexes satisfying Property A collapse to a point through CAT(0) subspaces. Our proof has two steps.
Firstly, because CAT(0)
spaces have a strongly convex metric, White's result given in \cite{white_1967} ensures that finite,
$3$-complexes endowed with a CAT(0) metric, have a $3$-simplex with a free $2$-dimensional
($1$-dimensional) face. So we may perform an elementary collapse on such complex.
The second step is to investigate whether the subcomplex obtained by performing an elementary
collapse on a
CAT(0) $3$-complex remains, at each step of the retraction, nonpositively curved. We will be able to
analyze whether such space still has locally a CAT(0) metric, only once we have found its
new local geodesic segments.

We start by characterizing the curvature of a $2$-simplex of a CAT(0) simplicial complex.

\begin{lemma}\label{3.1}
Let $K$ be a simplicial complex. If $|K|$
admits a CAT(0) metric $d$, then any $2$-simplex in $K$ is isometric to its comparison triangle in
$\mathbf{R}^{2}$.
\end{lemma}

\begin{proof}\label{3.2}

Let $\triangle (a,b,c)$ be a
$2$-simplex $\sigma$ of $K$ and let $d$ be a point on the edge $e =
[b,c]$. Let $\triangle (a',b',d')$ be a comparison triangle in
$\mathbf{R}^{2}$ for the geodesic triangle $\triangle (a,b,d)$ in $|K|$ and let $\triangle
(a',d',c')$ be a comparison triangle in $\mathbf{R}^{2}$ for the geodesic triangle
$\triangle (a,d,c)$ in $|K|$. We place the comparison triangles $\triangle
(a',b',d')$ and $\triangle (a',d',c')$ in different half-planes with
respect to the line $a'd'$ in $\mathbf{R}^{2}$.

Because any geodesic triangles in $|K|$ satisfies the CAT(0)
inequality and $d \in [b,c]$, we have \begin{center}$\pi = \angle_{d} (b,c) \leq
\angle_{d} (b,a) + \angle_{d} (a,c) \leq \angle_{d'} (b',a') +
\angle_{d'} (a',c')$.\end{center} So, since $\angle_{d'} (b',a') + \angle_{d'} (a',c')
\geq \pi$, Alexandrov's Lemma implies
\begin{center}$d_{\mathbf{R}^{2}}(a',d') \leq d(a,d)$.\end{center}
But $\triangle (a',b',d')$ is a comparison triangle for the geodesic triangle $\triangle
(a,b,d)$ in $|K|$ and therefore $d_{\mathbf{R}^{2}}(a',d') = d(a,d)$. Because one
equality in Alexandrov's Lemma implies the others, the following
equalities hold $\angle_{d'} (b',a') + \angle_{d'} (a',c') = \pi$,
$\angle_{b} (a,d) = \angle_{b'} (a',d')$, $\angle_{c} (a,d) =
\angle_{c'} (a',d')$ and $\angle_{a} (b,d) + \angle_{a} (d,c) =
\angle_{a'} (b',c')$. So the sum of the angles between the sides of
$\sigma$ equals $\pi$. Therefore, because $|K|$ has a convex
metric, the curvature of the $2$-simplex $\sigma$ equals
$\omega(\sigma) = \pi - \pi = 0$. So, since any
$2$-simplex in $K$ has curvature zero, any $2$-simplex in $K$ is
isometric to its comparison triangle in $\mathbf{R}^{2}$.

\end{proof}

We shall use the following lemmas frequently.

\begin{lemma}\label{3.3}
Let $(X,d)$ be a CAT(0) space. Then any path $c : [0,1] \rightarrow X$
in $X$ has a unique midpoint.
\end{lemma}

\begin{proof}\label{3.4}

Let $t \in [0,1]$ be such that $l(c|_{[0,t]}) = l(c|_{[t,1]}) =
\frac{\textstyle 1}{\textstyle 2}l(c|_{[0,1]})$. Because $X$ is a
CAT(0) space, for every $\varepsilon > 0$ there exists $\delta = \delta(\varepsilon)$ such that if
\begin{center}
$l(c|_{[0,t']}) = l(c|_{[t',1]}) = \frac{\textstyle 1}{\textstyle 2}l(c|_{[0,1]})
\leq \frac{\textstyle 1}{\textstyle 2}l(c|_{[0,1]}) + \delta $,\\
\end{center}
$t' \in [0,1]$, then $d(c(t), c(t')) < \varepsilon$. So, because
$d(c(t), c(t')) < \varepsilon$ for every $\varepsilon > 0$, $d(c(t),
c(t')) = 0$. The path $c$ has therefore a unique midpoint.

\end{proof}

\begin{lemma}\label{3.41}
Let $(X,d)$ be a CAT(0) space and let $p,q,s,t$ be four distinct points in $X$ such that
$\angle_{s}(p,t)
+ \angle_{s}(t,q) \geq \pi$. Then the following inequality holds
$d(p,s) + d(s,q) < d(p,t) + d(t,q)$.
\end{lemma}

\begin{proof}\label{3.42}

Let $\overline{\triangle} (p,t,s)$ be a comparison triangle in
$\mathbf{R}^{2}$ for the geodesic triangle $\triangle (p,t,s)$ in $X$ and let
$\overline{\triangle}
(q,t,s)$ be a comparison triangle in $\mathbf{R}^{2}$ for the geodesic triangle
$\triangle (q,t,s)$ in $X$. We place the comparison triangles
$\overline{\triangle} (p,t,s)$ and $\overline{\triangle} (q,t,s)$ in
different half-planes with respect to the line $\overline{t}\overline{s}$ in
$\mathbf{R}^{2}$.

By the CAT(0)
inequality, \begin{center} $ \angle_{s}(p,t) +
\angle_{s}(t,q) \leq \overline{\angle}_{s}(p,t) +
\overline{\angle}_{s}(t,q)$. \end{center} So, by hypothesis, it follows that

\begin{equation}\label{3.600}
\overline{\angle}_{s}(p,t) +
\overline{\angle}_{s}(t,q)
\geq \pi.
\end{equation}

If in \eqref{3.600} we have equality, taking into account that $d(t,s) \neq 0$, we get
$d_{\mathbf{R}^{2}}(\overline{p},
\overline{q}) = d_{\mathbf{R}^{2}}(\overline{p},
\overline{s}) + d_{\mathbf{R}^{2}}(\overline{s}, \overline{q}) <
d_{\mathbf{R}^{2}}(\overline{p}, \overline{t}) + d_{\mathbf{R}^{2}}(\overline{t},
\overline{q})$. So $d(p,s) + d(s,q) < d(p,t) + d(t,q)$.

\begin{figure}[ht]
  \vspace{-10pt}
  \begin{minipage}[b]{0.95\linewidth}
    \centering
    \raisebox{-0.75cm}{
      \includegraphics[height=4cm]{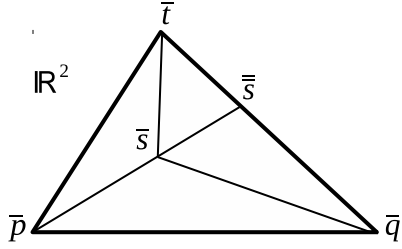}
    }
 \caption{Comparison triangles in $\mathbf{R}^{2}$}

    \label{fig:Pic1}
  \end{minipage}
\end{figure}

If the inequality in \eqref{3.600} is strict, the comparison triangles $\overline{\triangle}
(p,t,s)$ and $\overline{\triangle}
(q,t,s)$ in $\mathbf{R}^{2}$ are placed one with
respect to the other as in the figure above. Because the
curvature at any point in $\mathbf{R}^{2}$ equals zero, while any
Euclidean triangle has curvature zero, we get: $\overline{\angle}_{p}
(t,s) < \overline{\angle}_{p} (t,q)$ and
$\overline{\angle}_{q} (t,s) < \overline{\angle}_{q}
(t,p)$. The point $\overline{s}$ lies therefore in the interior of the
Euclidean triangle $\overline{\triangle} (t, p, q)$. We
consider a point $\overline{\overline{s}}$ on $[\overline{t}, \overline{q}]$ such that
$\overline{s}$ lies on $[\overline{p}, \overline{\overline{s}}]$. Thus

\begin{center}
$d_{\mathbf{R}^{2}}(\overline{p}, \overline{s}) +
d_{\mathbf{R}^{2}}(\overline{s}, \overline{q}) < $\\\hspace{30 mm}\\
$< d_{\mathbf{R}^{2}}(\overline{p},\overline{s}) +
d_{\mathbf{R}^{2}}(\overline{s},\overline{\overline{s}}) +
d_{\mathbf{R}^{2}}(\overline{\overline{s}}, \overline{q}) = $\\\hspace{30 mm}\\
$= d_{\mathbf{R}^{2}}(\overline{p},\overline{\overline{s}}) +
d_{\mathbf{R}^{2}}(\overline{\overline{s}},\overline{q})$.\hspace{30 mm}\\
\end{center}
Further

\begin{center}
$d_{\mathbf{R}^{2}}(\overline{p},\overline{\overline{s}}) +
d_{\mathbf{R}^{2}}(\overline{\overline{s}},\overline{q}) < $\\
\hspace{30 mm}\\
$< d_{\mathbf{R}^{2}}(\overline{p},\overline{t}) +
d_{\mathbf{R}^{2}}(\overline{t}, \overline{\overline{s}}) +
d_{\mathbf{R}^{2}}(\overline{\overline{s}},\overline{q}) =$\\
\hspace{30 mm}\\
$= d_{\mathbf{R}^{2}}(\overline{p},\overline{t}) +
d_{\mathbf{R}^{2}}(\overline{t},\overline{q})$.\\
\end{center}
Hence

\begin{center}
$d_{\mathbf{R}^{2}}(\overline{p},\overline{s}) +
d_{\mathbf{R}^{2}}(\overline{s},\overline{q}) <
d_{\mathbf{R}^{2}}(\overline{p},\overline{t}) +
d_{\mathbf{R}^{2}}(\overline{t},\overline{q})$.\\
\end{center}
So the following inequality holds in $X$:

\begin{center}
$d(p,s) + d(s,q) < d(p,t) + d(t,q)$.
\end{center}

\end{proof}

\begin{lemma}\label{3.43}
Let $(X,d)$ be a CAT(0) space and let $(s_{n})_{n \in \mathbf{N}}$ and $(t_{n})_{n \in
\mathbf{N}}$ be distinct sequences of points on a geodesic segment $e$ in $X$ such that $\underset{n
\rightarrow \infty}
\lim   d(s_{n},t_{n}) = 0$. Then for any point $p$ in $X$ which does not lie on the geodesic segment $e$, we
have $\underset{n \rightarrow
\infty}\lim
\angle _{p}(s_{n},t_{n}) = 0$.
\end{lemma}

\begin{proof}\label{3.44}

We consider the comparison triangle
$\overline{\triangle} (p,s_{n},t_{n})$ in $\mathbf{R}^{2}$ for the geodesic triangle
$\triangle (p,s_{n},t_{n})$ in $X$. By hypothesis, it follows that
\begin{center}
$\underset{n \rightarrow \infty} \lim
d_{\mathbf{R}^{2}}(\overline{s}_{n},\overline{t}_{n}) = 0$.\\
\end{center}
So
\begin{center}
$\underset{n \rightarrow \infty} \lim
\overline{\angle}_{p}(s_{n},t_{n}) = 0$.\\
\end{center}
Because $X$ is a CAT(0) space, we have
\begin{center}
$0 \leq \angle_{p}(s_{n},t_{n}) \leq
\overline{\angle}_{p}(s_{n},t_{n})$.\\
\end{center}
Hence \begin{center}$\underset{n \rightarrow \infty}\lim
\angle _{p}(s_{n},t_{n}) = 0$.
\end{center}

\end{proof}

We fix, for the remainder of the paper, the following notations.

Let $K$ be a finite simplicial $3$-complex endowed with a CAT(0) metric $d$ and satisfying Property A. Because $K$ has a
strongly convex metric, it has a
$3$-simplex $\sigma$ with a free $k$-dimensional face $\alpha$, $k \in \{1,2\}$. Let $K' = K
\setminus \{\alpha, \sigma\}$ be the subcomplex obtained by performing an elementary collapse on
$K$ and let $d'$ be the induced metric on $K'$.
Let $a,b,c$ and $d$ be the vertices of the $3$-simplex $\sigma$. Let $\tau_{1}, \tau_{2}$ and
$\tau_{3}$ be three $2$-dimensional faces of $\sigma$ different from the
free face of $\sigma$ (in case $\sigma$ has a free $2$-dimensional face). Let $\tau_{1} \cap
\tau_{2} = e_{1} = [a,b]$, $\tau_{1} \cap \tau_{3} =
e_{2}
= [a,d]$ and $\tau_{2} \cap \tau_{3} = e_{3} = [a,c]$ be three edges of $\sigma$ different from the
free face of $\sigma$ (in case $\sigma$ has a free $1$-dimensional face) or a face of the free face
of $\sigma$ (in case $\sigma$ has a free $2$-dimensional face). We denote by $r = \max
\{d(a,b), d(a,c), d(a,d), d(b,c),
d(b,d), d(c,d)\}$. We consider in $|K|$ a neighborhood of $a$ homeomorphic to a closed ball of
radius $r$, $U = \{x \in |K| \mid d(a,x) \leq r\}$. Note that $U$ endowed with the induced
metric is a CAT(0) space. Because $U$ is complete and it has a strongly convex metric, any two
points in $U$ are joined by a unique geodesic segment which
belongs to $U$. So any geodesic triangle with vertices at any
three points in $U$, belongs to $U$, and it satisfies the CAT(0)
inequality. Furthermore, $U$ being a CAT(0) space, any $2$-simplex in $U$ has curvature zero
and it is therefore isometric to its comparison triangle in
$\mathbf{R}^{2}$. We consider in $|K'|$ a neighborhood of $a$ homeomorphic to a closed
ball of radius $r$, $U' = \{x \in |K'| \mid d'(a,x) \leq r\}$. We note that $U' = U
\setminus \{\alpha, \sigma\}$. We consider in $U$ two distinct points $p$ and $q$ that do not belong
to $\sigma$ such that the geodesic segment $[p,q]$ intersects the interior of $\tau_{1}$ in $p_{1}$,
and the interior of $\tau_{2}$ in $q_{1}$.

\begin{figure}[h]
   \begin{center}
     \includegraphics[height=4.5cm]{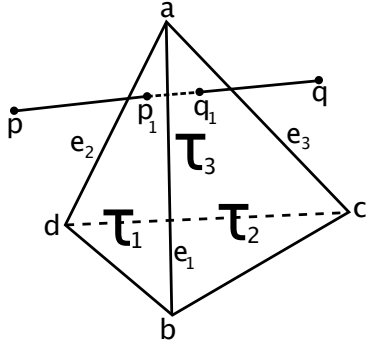}
        \caption{The $3$-simplex $\sigma$ in $K$ with the free $k$-dimensional face $\alpha$, $k \in
\{1,2\}$ intersected by the geodesic segment $[p,q]$}
 \end{center}
\end{figure}

We study further whether $K'$ still has locally a CAT(0) metric. Note that $K'$ inherits Property A from $K$. Namely, we will show that $U'$
is a CAT(0) space. We consider only the case when $K'$ is
obtained by pushing in an entire $3$-simplex with a free
face, by starting at its free face. It is important to note, however,
that the same result holds for any deformation retract of $|K|$ obtained by pushing in any
tetrahedron $\delta$ in $|K|$ such that one face of $\delta$ belongs to the free face of $K$. We
will be able to investigate whether any geodesic triangle in $U'$ satisfies the CAT(0)
inequality, only once we have found the new geodesic segments in $U'$.
We note that any two points in $U$ joined in $U$ by a segment which does not intersect the
interior of $\sigma$, are joined in $U'$ by a segment that coincides with
the segment that joins these points in $U$. So we still have to find the paths of shortest length in
$U'$ joining those pair of
points that are joined in $U$ by a segment that intersects the interior of $\sigma$. We shall be
concerned with this problem in the following corollary and five lemmas.

\begin{lemma}\label{3.7}
There exists a unique point $s$ on $e_{1}$ such that $\angle_{s}(a,p_{1}) =
\angle_{s}(b,q_{1})$ and $\angle_{s}(a,q_{1}) = \angle_{s}(b,p_{1})$. For such point $s$ we have
$\angle_{s}(p_{1},t) + \angle_{s}(t,q_{1}) = \pi$,
for any point $t$ on $e_{1}$ that differs from $s$. In particular,
the following inequality holds $d(p_{1},s) + d(s,q_{1}) < d(p_{1},t) + d(t,q_{1})$.
\end{lemma}

\begin{proof}\label{3.8}
We show first the existence of such point $s$.

We consider the path $c_{1} : [0,1] \rightarrow U$, $c_{1}(0) = a, c_{1}(1) = b, c_{1}(t) \in
e_{1}, \forall t \in (0,1)$,
i.e. the path $c_{1}$ is the edge $e_{1}$.

Note that for any $t \in (0,1)$,
\begin{equation}\label{3.100}
\angle_{c_{1}(t)}(a,p_{1}) + \angle_{c_{1}(t)}(p_{1},b) = \angle_{c_{1}(t)}(a,q_{1}) +
\angle_{c_{1}(t)}(q_{1},b) = \pi.\end{equation} So
\begin{equation}\label{3.101}
\angle_{c_{1}(t)}(a,p_{1}) + \angle_{c_{1}(t)}(p_{1},b) + \angle_{c_{1}(t)}(b,q_{1}) +
\angle_{c_{1}(t)}(q_{1},a) = 2\pi.
\end{equation}

We call the points $c_{1}(t), t \in [0,1]$ such that \begin{center}$\angle_{c_{1}(t)}(p_{1},a) +
\angle_{c_{1}(t)}(a,q_{1}) >
\angle_{c_{1}(t)}(q_{1},b) + \angle_{c_{1}(t)}(b,p_{1})$,\end{center} points of type I. Relation \eqref{3.101}
implies
that
for
any point of type I we have: $\angle_{c_{1}(t)}(p_{1},b) + \angle_{c_{1}(t)}(b,q_{1}) < \pi$ and
$\angle_{c_{1}(t)}(q_{1},a) + \angle_{c_{1}(t)}(a,p_{1}) > \pi$. Let $c_{1}(t_{1})$, $t_{1} \in
[0,1]$
be the point of type I on $c_{1}$ such
that $d(a,c_{1}(t_{1})) < d(a,c_{1}(t_{1}'))$, for any
point of type I $c_{1}(t_{1}')$, $t_{1}' \in [0,1], t_{1}' \neq t_{1}$.

We call the points $c_{1}(t), t \in [0,1]$ such that \begin{center}$\angle_{c_{1}(t)}(p_{1},a) +
\angle_{c_{1}(t)}(a,q_{1}) <
\angle_{c_{1}(t)}(q_{1},b) + \angle_{c_{1}(t)}(b,p_{1})$,\end{center} points of type II. By \eqref{3.101},
for any point of type II we have: $\angle_{c_{1}(t)}(p_{1},b) + \angle_{c_{1}(t)}(b,q_{1}) > \pi$
and
$\angle_{c_{1}(t)}(q_{1},a) + \angle_{c_{1}(t)}(a,p_{1}) < \pi$. Let $c_{1}(t_{2})$, $t_{2} \in
[0,1]$
be the point of type II on $c_{1}$ such
that $d(b,c_{1}(t_{2})) < d(b,c_{1}(t_{2}'))$, for any
point of type II $c_{1}(t_{2}')$, $t_{2}' \in [0,1], t_{2}' \neq t_{2}$.

We call the points $c_{1}(t), t \in [0,1]$ such that \begin{center}$\angle_{c_{1}(t)}(p_{1},a) +
\angle_{c_{1}(t)}(a,q_{1}) =
\angle_{c_{1}(t)}(q_{1},b) + \angle_{c_{1}(t)}(b,p_{1})$,\end{center} points of type III. Relation \eqref{3.101}
implies that for
any point of type III we have: $\angle_{c_{1}(t)}(p_{1},a) +
\angle_{c_{1}(t)}(a,q_{1}) =
\angle_{c_{1}(t)}(q_{1},b) + \angle_{c_{1}(t)}(b,p_{1}) = \pi$. Note that, by \eqref{3.100},
any point of type III fulfills the following $\angle_{c_{1}(t)}(a,p_{1}) =
\angle_{c_{1}(t)}(b,q_{1})$ and $\angle_{c_{1}(t)}(a,q_{1}) = \angle_{c_{1}(t)}(b,p_{1})$.

Suppose that there are no points of type III on $c_{1}$. Any point on
$c_{1}$ is therefore either a point of type I or a point of type
II.

We define the mapping $\rm{mid}: c_{1}[0,1] \times c_{1}[0,1] \rightarrow c_{1}[0,1]$
by \begin{center}$\forall t_{1}, t_{2} \in [0,1]$, $\rm{mid}(c_{1}(t_{1}), c_{1}(t_{2})) =
c_{1}(t)$,
$t \in [0,1]$,\end{center} where \begin{center}$l(c_{1}|_{[t_{1},t]}) = l(c_{1}|_{[t,t_{2}]}) =
\frac{\textstyle 1}{\textstyle 2}l(c_{1}|_{[t_{1},t_{2}]})$.\end{center} Because $U$
is a CAT(0) space, Lemma \ref{3.3} guarantees that the path $c_{1}$ has a unique
midpoint. The mapping $\rm{mid}$ is therefore well-defined.

We define the sequence $(s_{n})_{n \in \mathbf{N}}$ of tuples $(s'_{n}, s''_{n})$ as follows:
\begin{displaymath}
  \begin{tabular}{ l c l c l }
   \multicolumn{5}{ l }{-- the elements $s'_{n}$ are points of type I;}\\
   \hbox{} \\
   \multicolumn{5}{ l }{-- the elements $s''_{n}$ are points of type II;}\\
   \hbox{} \\
   -- $s_{0}$&=&$(s'_{0}, s''_{0})$&=&$(c_{1}(t_{1}), c_{1}(t_{2}))$; \\
   \hbox{} \\
   -- $s_{1}$&=&$(s'_{1}, s''_{1})$&=&
     $\left\{
       \begin{array}{ll}
         (s'_{0}, \rm{mid}(s'_{0}, s''_{0})), & \hbox{if } \rm{mid}(s'_{0}, s''_{0}) \hbox{ is a
point of
type
II;} \\
         \hbox{} \\
         (\rm{mid}(s'_{0}, s''_{0}), s''_{0}), & \hbox{if } \rm{mid}(s'_{0}, s''_{0}) \hbox{ is a
point of
type
I;}
       \end{array}
     \right.$ \\
   \hspace{2mm}\hbox{...} \\
   -- $s_{n}$&=&$(s'_{n}, s''_{n})$&=&
     $\left\{
       \begin{array}{ll}
         (s'_{n-1}, \rm{mid}(s'_{n-1}, s''_{n-1})), & \hbox{if } \rm{mid}(s'_{n-1}, s''_{n-1})
\hbox{ is a
point
of type II;} \\
         \hbox{} \\
         (\rm{mid}(s'_{n-1}, s''_{n-1}), s''_{n-1}), & \hbox{if } \rm{mid}(s'_{n-1}, s''_{n-1})
\hbox{ is a
point
of type I.}
       \end{array}
     \right.$
  \end{tabular}
\end{displaymath}

\begin{figure}[h]
   \begin{center}
     \includegraphics[height=4cm]{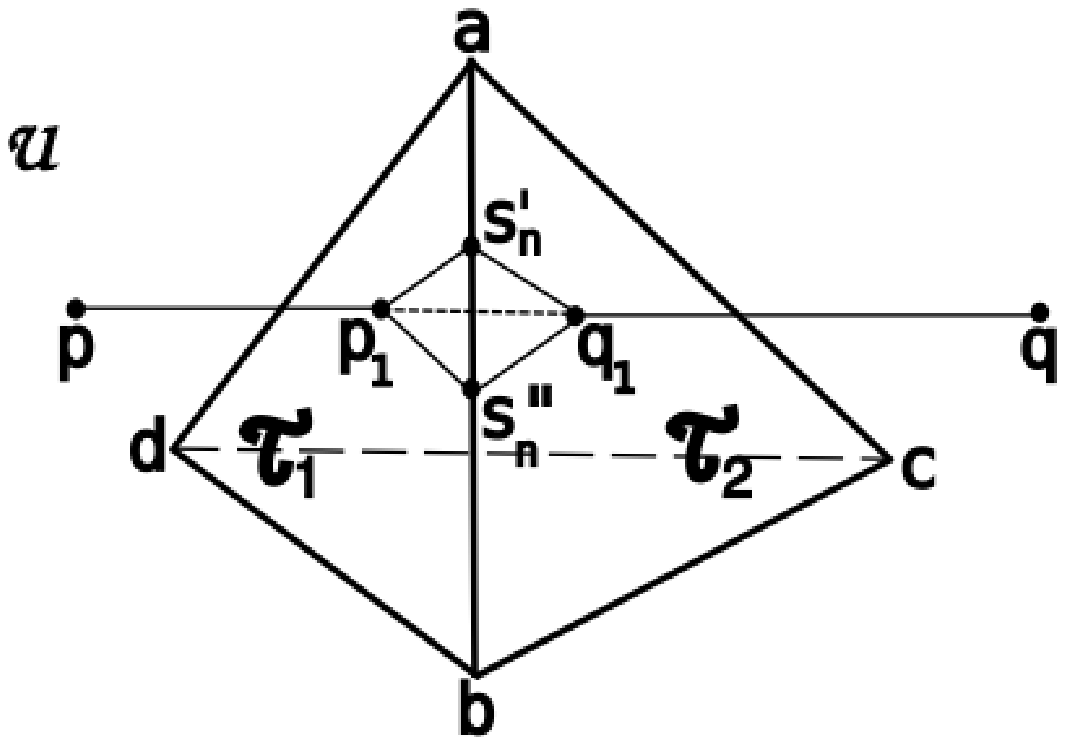}
 \caption{The segment $[p,q]$ that intersects the interior of $\sigma$\\$s_{n}'$ is a point of type
I on $e_{1}$\\$s_{n}''$ is a point of type II on $e_{1}$}
 \end{center}
\end{figure}

Let $s_{n}' = c_{1}(t_{n}')$ be a point of type I on $c_{1}$ and let
$s_{n}'' = c_{1}(t_{n}'')$ be a point of type II on $c_{1}$, $n \geq 1$ such that the position of
$s_{n}'$ with respect to $s_{n}''$ on
the edge $e_{1}$ is as in the figure below.
Because \begin{displaymath}
 \begin{tabular}{ l c l c l }
$\left\{
\begin{array}{ll}
         l(c_{1}|_{[t'_{n},t''_{n}]}) = \frac{\textstyle 1}{\textstyle 2^{n}}l(c_{1}|_{[0,1]}),  &
\hbox{if
}  t_{n}' < t_{n}'', \\
         \hbox{} \\
         l(c_{1}|_{[t''_{n},t'_{n}]}) = \frac{\textstyle 1}{\textstyle 2^{n}}l(c_{1}|_{[0,1]}), &
\hbox{if
}
t_{n}'' \leq t_{n}',
\end{array}
\right.$ \\
 \end{tabular}
\end{displaymath}
we have
\begin{displaymath}
 \begin{tabular}{ l c l c l }
$\left\{
\begin{array}{ll}
         \underset{n \rightarrow \infty}\lim l(c_{1}|_{[t'_{n},t''_{n}]}) = 0,  & \hbox{if }  t_{n}'
< t_{n}'', \\
         \hbox{} \\
         \underset{n \rightarrow \infty}\lim l(c_{1}|_{[t''_{n},t'_{n}]}) = 0, & \hbox{if } t_{n}''
\leq
t_{n}'.
\end{array}
\right.$ \\
 \end{tabular}
\end{displaymath}
There exists a unique geodesic segment in $U$ joining $s_{n}' =
c_{1}(t_{n}')$ to $s_{n}'' = c_{1}(t_{n}'')$ whose length equals $d(s_{n}',
s_{n}'')$. Because
\begin{displaymath}
 \begin{tabular}{ l c l c l }
$\left\{
\begin{array}{ll}
         0 \leq d(s_{n}', s_{n}'') \leq l(c_{1}|_{[t_{n}', t_{n}'']}),  & \hbox{if }  t_{n}' <
t_{n}'', \\
         \hbox{} \\
         0 \leq d(s_{n}', s_{n}'') \leq l(c_{1}|_{[t_{n}'', t_{n}']}), & \hbox{if } t_{n}'' \leq
t_{n}',
\end{array}
\right.$ \\
 \end{tabular}
\end{displaymath}
we get
\begin{center}
$\underset{n \rightarrow \infty} \lim   d(s_{n}', s_{n}'') = 0.$
\end{center}
Hence, by Lemma \ref{3.43}, \begin{equation}\label{3.300}\underset{n \rightarrow \infty}\lim
\angle _{p_{1}}(s_{n}', s_{n}'') = 0
\end{equation} and \begin{equation}\label{3.400}
\underset{n \rightarrow \infty}\lim
\angle _{q_{1}}(s_{n}', s_{n}'') = 0.
\end{equation}

On the other hand, because the geodesic triangles $\triangle(p_{1}, s_{n}',s_{n}'')$ and
$\triangle(q_{1},
s_{n}',s_{n}'')$
belong to $2-$simplices of curvature zero, we have
\begin{center}
$\angle _{p_{1}}(s_{n}', s_{n}'') + \angle _{s_{n}'}(p_{1}, s_{n}'') + \angle _{s_{n}''}(s_{n}', p_{1})  + \angle
_{q_{1}}(s_{n}', s_{n}'') + \angle _{s_{n}''}(q_{1}, s_{n}') + \angle _{s_{n}'}(s_{n}'',
q_{1}) =
2\pi$.
\end{center}
Because $s_{n}'$ is a point of type I, while $s_{n}''$ is a point of type II which lie one with respect
to the other on the edge $e_{1}$ as in the figure above, we have \begin{center}
$\angle _{s_{n}'}(p_{1}, s_{n}'') + \angle _{s_{n}'}(s_{n}'',
q_{1}) < \pi$
\end{center}
and \begin{center}
$\angle _{s_{n}''}(p_{1}, s_{n}') + \angle _{s_{n}''}(s_{n}', q_{1}) < \pi$.
\end{center}
The above three relations imply that  \begin{center}
$\angle _{p_{1}}(s_{n}', s_{n}'') + \angle
_{q_{1}}(s_{n}', s_{n}'') > 0$.
\end{center}
So, since any Alexandrov angle is a value in the interval $[0,\pi]$, either \begin{center}
$\underset{n \rightarrow \infty}\lim
\angle _{p_{1}}(s_{n}', s_{n}'') \neq 0$ \\
\end{center}
or \begin{center}
$\underset{n \rightarrow \infty}\lim
\angle _{q_{1}}(s_{n}', s_{n}'') \neq 0$ \\
\end{center}
or \begin{center}
$\underset{n \rightarrow \infty}\lim
\angle _{p_{1}}(s_{n}', s_{n}'') \neq 0$ and $\underset{n \rightarrow \infty}\lim
\angle _{q_{1}}(s_{n}', s_{n}'') \neq 0$. \\
\end{center}
Thus, according either to \eqref{3.300} or to \eqref{3.400} or to both, we have reached a
contradition. So there exist points of type III on $e_{1}$.

We show further that there exists a unique point of type III on $e_{1}$. Suppose, on the contrary,
there exist two points of
type III, say $s_{1}$ and $s_{2}$, on $e_{1}$. We assume that the position of $s_{1}$ with respect
to $s_{2}$ on the edge $e_{1}$ is as in the figure below.

\begin{figure}[h]
   \begin{center}
     \includegraphics[height=3cm]{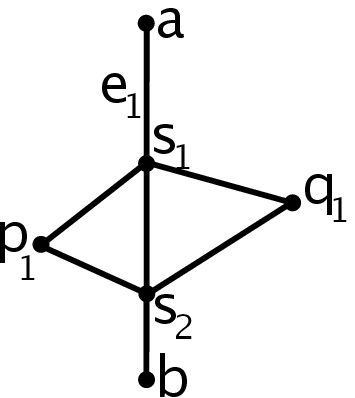}
        \caption{There exists a unique point of type III on $e_{1}$}
 \end{center}
\end{figure}
So \begin{center}
$\angle_{s_{1}}(p_{1},s_{2}) +
\angle_{s_{1}}(s_{2},q_{1}) = \pi$ \end{center} and \begin{center} $\angle_{s_{2}}(p_{1},s_{1}) +
\angle_{s_{2}}(s_{1},q_{1}) = \pi$. \end{center} Note that, since that the geodesic triangles
$\triangle(p_{1}, s_{1}, s_{2})$ and $\triangle(q_{1}, s_{1}, s_{2})$ belong to
$2$-simplices of curvature zero, we have
\begin{center}
$\angle_{p_{1}}(s_{1},s_{2}) + \angle_{s_{1}}(p_{1},s_{2}) + \angle_{s_{2}}(p_{1},s_{1}) +
\angle_{q_{1}}(s_{1},s_{2}) + \angle_{s_{1}}(q_{1},s_{2}) + \angle_{s_{2}}(q_{1},s_{1}) =
2\pi$.\end{center}
The above three relations imply that: \begin{center}
$\angle_{p_{1}}(s_{1},s_{2}) = 0$ \end{center} and \begin{center}
$\angle_{q_{1}}(s_{1},s_{2}) = 0$.
\end{center} Because the points $s_{1}$ and $s_{2}$ belong to $2$-simplices that are isometric
to geodesic triangles in $\mathbf{R}^{2}$, these relations ensure that $s_{1} = s_{2}$. So there
exists a
unique point, say $s$, on $e_{1}$ such that
$\angle_{s}(p_{1},t) + \angle_{s}(t,q_{1}) = \pi$, for
any point
$t$ on $e_{1}$ that differs from $s$.
Then, according to Lemma \ref{3.41},
the following inequality holds $d(p_{1},s) + d(s,q_{1}) < d(p_{1},t) + d(t,q_{1})$.

\end{proof}

The aim of the following lemma is to show that a relation similar to the one proven in the lemma above for the pair $p_{1}, q_{1}$, holds for the
pair of points $p,q$ as well. We prove this by showing that such relation holds, in fact, for
any pair of points on the geodesic segment $[p,q]$ such that one point of the pair lies on
$[p_{1},p]$, while the other one lies on $[q_{1},q]$.
The lemma follows due to the fact that the points $p,p_{1},q_{1}$ and $q$ lie, in this
order, on the same geodesic segment in a CAT(0) space.
It is important to keep in mind that any Alexandrov angle is a
value in the interval $[0,\pi]$.

\begin{lemma}\label{3.9}
Let $s$ be a point on $e_{1}$ such that $\angle_{s}(a,p_{1}) =
\angle_{s}(b,q_{1})$ and $\angle_{s}(a,q_{1}) = \angle_{s}(b,p_{1})$. Then $\angle_{s}(p,t) +
\angle_{s}(t,q) \geq \pi$ for any point $t$ on $e_{1}$ that differs from $s$.
In particular the following inequality holds:
$d(p,s) + d(s,q) < d(p,t) + d(t,q)$.
\end{lemma}

\begin{proof}\label{3.10}
By Lemma \ref{3.7}, the point $s$ exists, it is unique and it fulfills the following relation
\begin{center}$\angle_{s}(p_{1},t) + \angle_{s}(t,q_{1}) = \pi$\end{center} for any point $t$ on
$e_{1}$ that differs from $s$.
In particular,
\begin{equation}\label{3.700} \angle_{s}(p_{1},a) + \angle_{s}(a,q_{1}) = \pi.\end{equation}

We construct a sequence of points $(p_{n}^{\ast})_{n \in \mathbf{N}}$ such that $p_{0}^{\ast} =
p, p_{n}^{\ast} \in [p,p_{1}],$ \begin{center}$\underset{n \rightarrow \infty} \lim
d(p_{1},p_{n}^{\ast}) = 0$.\end{center}
Lemma \ref{3.43} implies \begin{equation}\label{3.2000}\underset{n \rightarrow \infty} \lim
\angle_{s}(p_{1},p_{n}^{\ast}) =
0.\end{equation}
Similarly, we construct a sequence of points $(q_{n}^{\ast})_{n \in \mathbf{N}}$ such that
$q_{0}^{\ast}
= q, q_{n}^{\ast} \in [q,q_{1}],$ \begin{center}$\underset{n \rightarrow \infty} \lim
d(q_{1},q_{n}^{\ast}) =
0$.\end{center} Lemma \ref{3.43} implies \begin{equation}\label{3.2100}\underset{n \rightarrow
\infty} \lim
\angle_{s}(q_{1},q_{n}^{\ast}) = 0.\end{equation}

\begin{figure}[h]
   \begin{center}
     \includegraphics[height=5cm]{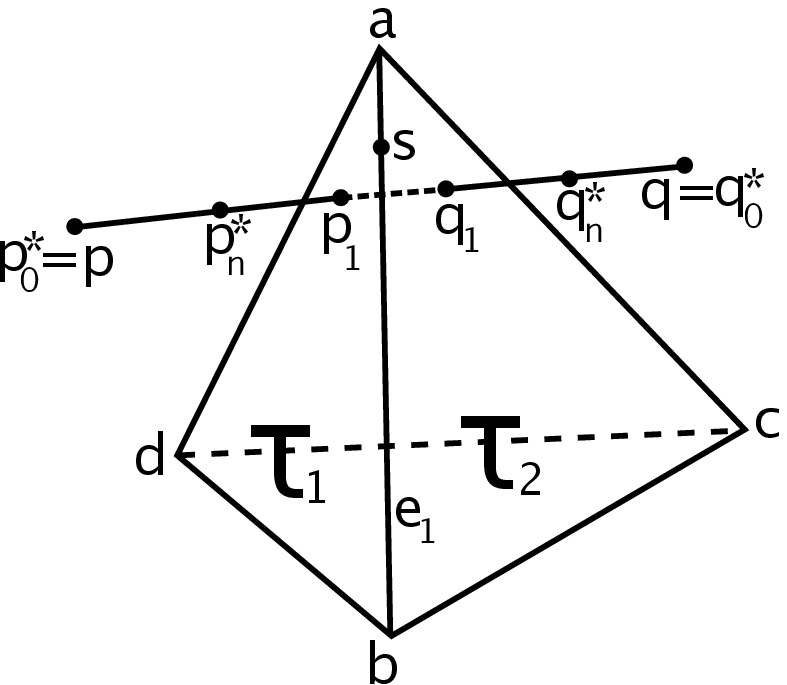}
        \caption{The sequence of points $(p_{n}^{\ast})_{n \in \mathbf{N}}$ on $[p,p_{1}]$\\
The sequence of points $(q_{n}^{\ast})_{n \in \mathbf{N}}$ on $[q,q_{1}]$}
 \end{center}
\end{figure}

Note that, since $p_{n}^{\ast} \in [p,p_{1}]$ and $U$ is a CAT(0) space, we have
\begin{center}$\angle_{p_{n}^{\ast}}(p,p_{1})
= \pi$.\end{center} Also note that \begin{center}$\angle_{s}(b,p_{1}) \leq
\angle_{s}(b,p_{n}^{\ast})$.\end{center} Thus, since \begin{center}$\angle_{s}(b,p_{1}) +
\angle_{s}(p_{1},a) = \pi,$\end{center} while \begin{center}$\angle_{s}(b,p_{n}^{\ast}) +
\angle_{s}(p_{n}^{\ast},a) = \pi$,\end{center} it follows that
\begin{center}$\angle_{s}(p_{n}^{\ast},a) \leq \angle_{s}(p_{1},a)$. \end{center}
Hence, since \begin{center}$\angle_{s}(p_{n}^{\ast},a) \leq
\angle_{s}(p_{1},a) \leq
\angle_{s}(p_{1},p_{n}^{\ast}) +
\angle_{s}(p_{n}^{\ast},a),$\end{center} by  \eqref{3.2000}, we have \begin{center}$\underset{n
\rightarrow \infty} \lim \angle_{s}(p_{n}^{\ast},a) =
\angle_{s}(p_{1},a).$\end{center} Similarly, relation \eqref{3.2100} implies that
\begin{center}$\underset{n
\rightarrow \infty} \lim \angle_{s}(a,q_{n}^{\ast}) =
\angle_{s}(a,q_{1}).$\end{center}
So \begin{equation}\label{3.1800}\underset{n \rightarrow \infty} \lim (\angle_{s}(p_{n}^{\ast},a) +
\angle_{s}(a,q_{n}^{\ast})) =
\angle_{s}(p_{1},a) + \angle_{s}(a,q_{1}).\end{equation}
Suppose that for any $n,$ \begin{center}$\angle_{s}(p_{n}^{\ast},a) +
\angle_{s}(a,q_{n}^{\ast}) < \pi$. \end{center}
The relations \eqref{3.700} and \eqref{3.1800} imply in this case a contradiction. So
there
exists
$m_{0} \in \mathbf{N}$ such that $p_{m_{0}}^{\ast} \in [p_{0}^{\ast},p_{1}]$,
$q_{m_{0}}^{\ast} \in [q_{0}^{\ast},q_{1}]$ and \begin{center}$\angle_{s}(p_{m_{0}}^{\ast},a) +
\angle_{s}(a,q_{m_{0}}^{\ast}) \geq \pi$. \end{center}
\begin{figure}[h]
   \begin{center}
     \includegraphics[height=2cm]{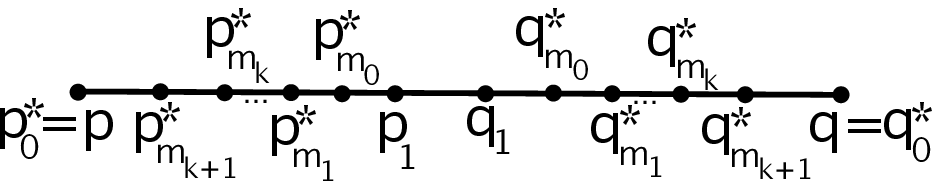}
        \caption{}
 \end{center}
\end{figure}

We argue by induction on $m_{k}$. The case $k=0$ is discussed above.
Replacing the pair $p_{1}, q_{1}$ by the pair
$p_{m_{0}}^{\ast}, q_{m_{0}}^{\ast}$ and arguing as above, it follows that there
exists
$m_{1} \in \mathbf{N}^{\ast}$ such that $p_{m_{1}}^{\ast} \in [p,p_{m_{0}}^{\ast}]$,
$q_{m_{1}}^{\ast} \in [q,q_{m_{0}}^{\ast}]$ and \begin{center}$\angle_{s}(p_{m_{1}}^{\ast},a) +
\angle_{s}(a,q_{m_{1}}^{\ast}) \geq \pi$. \end{center}
We proceed with the second step of the induction. Suppose there exists $m_{k} \in
\mathbf{N}^{\ast}$, $k \in \mathbf{N}^{*}$ such that $p_{m_{k}}^{\ast} \in [p,p_{m_{k-1}}^{\ast}],
q_{m_{k}}^{\ast} \in [q,q_{m_{k-1}}^{\ast}]$ and
\begin{center}$\angle_{s}(p_{m_{k}}^{\ast},a) +
\angle_{s}(a,q_{m_{k}}^{\ast}) \geq \pi.$\end{center}
Replacing the pair $p_{m_{k-1}}^{*}, q_{m_{k-1}}^{*}$ by the pair $p_{m_{k}}^{*},
q_{m_{k}}^{*}$ and arguing as for the case $k=0$, it similarly follows that there exists $m_{k+1}
\in
\mathbf{N}^{\ast}$
such
that $p_{m_{k+1}}^{\ast} \in [p,p_{m_{k}}^{\ast}], q_{m_{k+1}}^{\ast} \in [q,q_{m_{k}}^{\ast}]$ and
\begin{equation}\label{3.1500}\angle_{s}(p_{m_{k+1}}^{\ast},a) +
\angle_{s}(a,q_{m_{k+1}}^{\ast}) \geq \pi\end{equation}which concludes the second step of the
induction.

Note that, since $(p_{m_{k}}^{\ast})_{k \in \mathbf{N}}$ is a
sequence of points on $[p,p_{1}]$ such that $p_{m_{k}}^{\ast} \in [p,p_{m_{k-1}}^{\ast}]$, we have
\begin{center}$\underset{k
\rightarrow \infty} \lim
d(p_{m_{k}}^{\ast},p) =
0.$ \end{center}
Similarly note that $(q_{m_{k}}^{\ast})_{k \in \mathbf{N}}$ is a
sequence of points on $[q,q_{1}]$ such that \begin{center} $\underset{k \rightarrow \infty} \lim
d(q_{m_{k}}^{\ast},q) =
0.$ \end{center} Lemma \ref{3.43} further implies that
\begin{equation}\label{3.2200}\underset{k  \rightarrow \infty} \lim
\angle_{s}(p_{m_{k}}^{\ast},p) =
0 \end{equation} and \begin{equation}\label{3.2300} \underset{k \rightarrow \infty} \lim
\angle_{s}(q_{m_{k}}^{\ast},q) =
0. \end{equation}
Note that \begin{center}$\angle_{s}(p_{m_{k}}^{\ast},a) \leq \angle_{s}(p_{m_{k}}^{\ast},p) +
\angle_{s}(p,a)$,\end{center} while \begin{center}$\angle_{s}(q_{m_{k}}^{\ast},a) \leq
\angle_{s}(q_{m_{k}}^{\ast},q) +
\angle_{s}(q,a)$.\end{center}
So, by \eqref{3.2200} and \eqref{3.2300}, we get
\begin{center}$\underset{k \rightarrow \infty} \lim(\angle_{s}(p_{m_{k}}^{\ast},a) +
\angle_{s}(a,q_{m_{k}}^{\ast}))  \leq \angle_{s}(p,a) +
\angle_{s}(a,q)$.\end{center} Hence, by \eqref{3.1500}, \begin{center}$\angle_{s}(p,a) +
\angle_{s}(a,q) \geq \pi$.\end{center} Arguing similarly, one can show that
\begin{center}$\angle_{s}(p,t) +
\angle_{s}(t,q) \geq \pi$\end{center} for any $t$ on $e_{1}$ that differs from $s$. Lemma \ref{3.41}
ensures that \begin{center}$d(p,s) +
d(s,q) < d(p,t) + d(t,q).$\end{center}

\end{proof}

\begin{remark}
Note that in the proof of the previous lemma, the fact that the point $t$ lies on the edge $e_{1}$
does not influence the proof in
any way. So a similar result holds for the case when $t \in |U|, t \notin e_{1}$.
Moreover,
according to the hypothesis of Lemma \ref{3.9}, $s$ is the unique point on $e_{1}$ such that
$\angle_{s}(a,p_{1}) =
\angle_{s}(b,q_{1})$ and $\angle_{s}(a,q_{1}) = \angle_{s}(b,p_{1})$ and hence $\angle_{s}(p_{1},t)
+ \angle_{s}(t,q_{1}) =
\pi$. Note that a slightly modified
hypothesis in Lemma \ref{3.9},
namely $\angle_{s}(p_{1},t) + \angle_{s}(t,q_{1})
\geq \pi$, would imply the same result. Furthermore, note that the particular choice of the point $s$
does not influence the proof of Lemma \ref{3.9} either. Hence, for any $l \in e_{1}$ for whom
$\angle_{l}(p_{1},t) + \angle_{l}(t,q_{1}) \geq
\pi$ holds, Lemma \ref{3.9}
ensures the following
corollary.

\end{remark}

\begin{corollaryA}
For any $t \in |U|$ and for any $l \in e_{1},$ if $\angle_{l}(p_{1},t) + \angle_{l}(t,q_{1}) \geq
\pi$ then $\angle_{l}(p,t) + \angle_{l}(t,q) \geq
\pi$.
\end{corollaryA}

We summarize the basic ideas behind the proof of the above results.
For any $t \in |U|$ and for any $l \in e_{1},$ the inequality $\angle_{l}(p_{1},t) +
\angle_{l}(t,q_{1}) \geq
\pi$ is fulfilled by the pair of points $p_{1},q_{1}$ due to the fact that
such points lie on $2$-simplices that are isometric to their comparison triangles in Euclidean
plane. Furthermore,
such inequality is inherited by those pair of points on $[p,q]$ for whom one point of the pair lies
on $[p,p_{1}]$, while the other point lies on $[q,q_{1}]$.

\begin{lemma}\label{3.11}
Let $c: [0,1] \rightarrow U$ be a path in $U$ joining $p$ to
$q$ that does not intersect $\sigma$. Then there exists a point $m$ on $c$ such
that neither the segment $[p,m]$ in $U$ nor the segment $[q,m]$ in $U$ intersects the interior of
$\sigma$.
\end{lemma}

\begin{proof}\label{3.12}

We call the points $c(t)$ such that the segment $[q,c(t)]$
intersects the interior of $\sigma$ and the segment $[p,c(t)]$ does not intersect
the interior of $\sigma$, $t \in [0,1]$, points of type I.

We call the points $c(t)$ such that the segment $[p,c(t)]$
intersects the interior of $\sigma$, points of type II. Notice that if $c(t)$ is a
point of type II, $t \in [0,1]$, then the segment $[q,c(t)]$ might also intersect
the interior of $\sigma$.

We call the points $c(t)$ such that the segments $[q,c(t)]$ and
$[p,c(t)]$ do not intersect the interior of $\sigma$, $t \in [0,1]$, points of type III.

Suppose that there are no points of type III on the path $c$. Any point on
$c$ is therefore either a point of type I or a point of type
II. Thus, for any $t \in [0,1]$, at least one of the segments
$[q,c(t)]$ and $[p,c(t)]$ intersects the interior of $\sigma$.

Considering $(s'_{0}, s''_{0})=(p, q)$,
we define as in Lemma \ref{3.7} a
sequence $(s_{n})_{n \in \mathbf{N}} = (s'_{n}, s''_{n})$ of tuples such that
$s_{n}'$ is a point of type I on $c$ whereas $s_{n}''$ is a point of type II on $c$, $n \geq
1$. Assume that the position of $s_{n}'$ with respect to $s_{n}''$ on the path $c$ is as in the
figure below. Arguments similar to those in Lemma \ref{3.7} ensure that
\begin{equation}\label{3.1000}
\underset{n \rightarrow \infty} \lim   d(s_{n}', s_{n}'') = 0.
\end{equation}

We denote by $p_{n}'$ the intersection point of $[p,s_{n}']$ with $\tau_{1}$, and by $q_{n}'$ the
intersection point of $[p,s_{n}']$ with $\tau_{2}$. We denote by $p_{n}''$ the intersection point
of $[q,s_{n}'']$ with $\tau_{1}$, and by $q_{n}''$ the
intersection point of $[q,s_{n}'']$ with $\tau_{2}$.

\begin{figure}[h]
   \begin{center}
     \includegraphics[height=5cm]{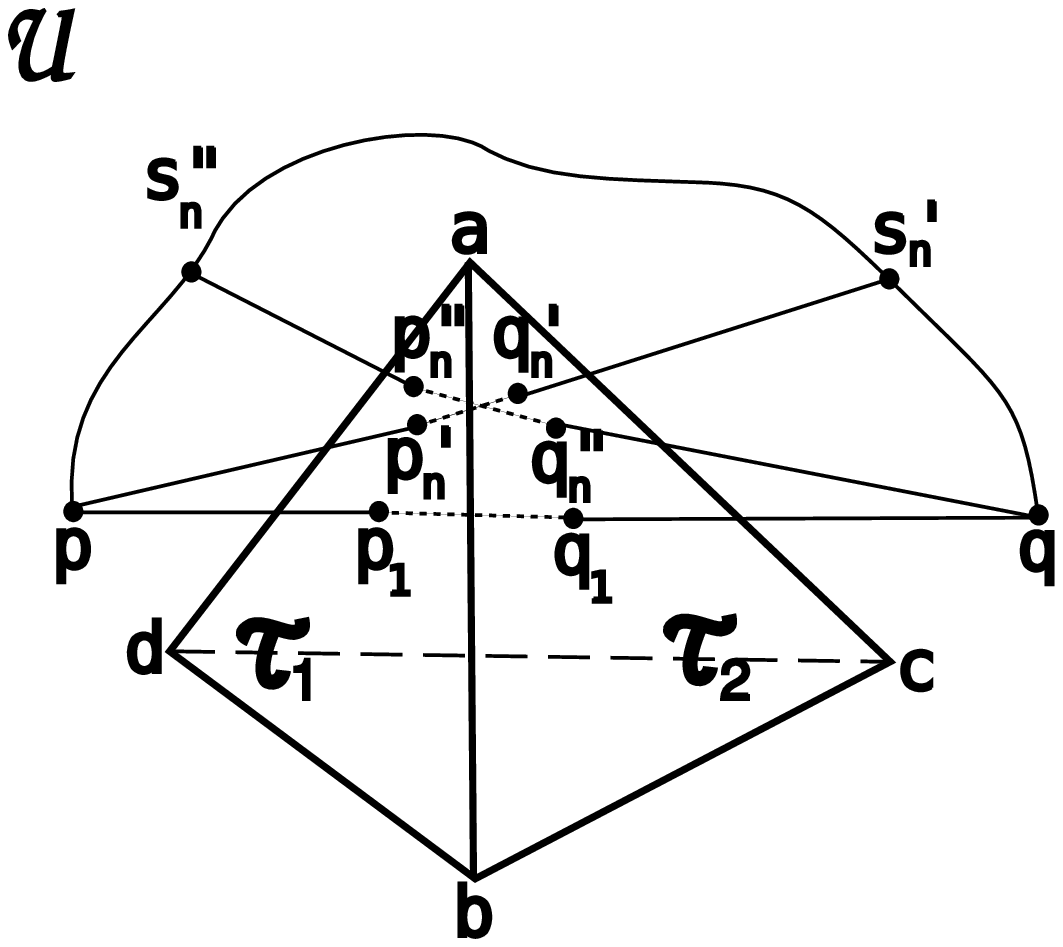}
        \caption{The path $c$ that connects $p$ to $q$ in $U$ without intersecting
$\sigma$\\$s_{n}'$ is a point of type I on $c$\\$s_{n}''$ is a point of type II on $c$}
 \end{center}
\end{figure}

Because $U$ is complete and a CAT(0) space, it satisfies the CAT(0) $4-$point condition. So the
$4$-tuple of points $(p_{n}', q_{n}',s_{n}', s_{n}'')$ in $U$ has a subembedding
$(\overline{p}_{n}', \overline{q}_{n}',\overline{s}_{n}', \overline{s}_{n}'')$ in
$\mathbf{R}^{2}$. Thus \begin{equation}\label{3.1400} d(s_{n}', p_{n}') \leq
d_{\mathbf{R}^{2}}(\overline{s}_{n}',
\overline{p}_{n}'). \end{equation}
Let $Q$ denote the quadrilateral in $\mathbf{R}^{2}$spanned by the vertices $\overline{p}_{n}',
\overline{q}_{n}',\overline{s}_{n}'$ and $\overline{s}_{n}''$.

\begin{figure}[h]
   \begin{center}
     \includegraphics[height=2.9cm]{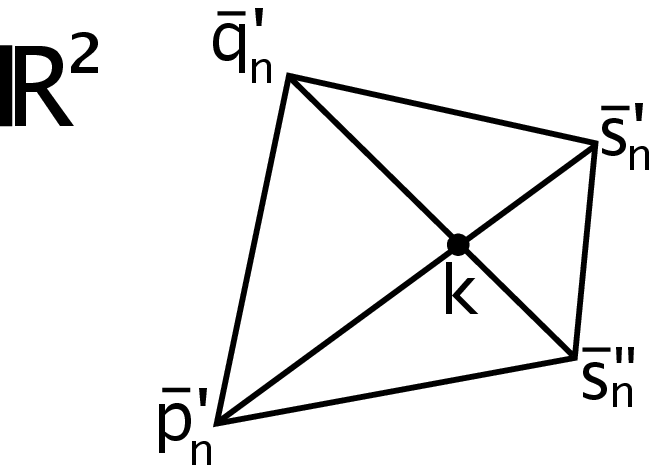}
 \caption{$Q$ is convex}
 \end{center}
\end{figure}

Suppose first that $Q$ is convex. Let $k$ denote the intersection point of its diagonals.
By \eqref{3.1000},
\begin{center}$\underset{n \rightarrow \infty} \lim d_{\mathbf{R}^{2}}(\overline{s}_{n}',
\overline{s}_{n}'') =
0.$\end{center}
Thus \begin{center}$\underset{n \rightarrow \infty} \lim \overline{\angle}_{k}(s_{n}', s_{n}'') =
0$\end{center}
and therefore
\begin{center}$\underset{n \rightarrow \infty} \lim \overline{\angle}_{k}(p_{n}', q_{n}') =
0.$\end{center}
So
\begin{center}$\underset{n \rightarrow \infty} \lim d_{\mathbf{R}^{2}}(\overline{p}_{n}',
\overline{q}_{n}') =
0$\end{center}
and hence
\begin{equation}\label{3.1100}\underset{n \rightarrow \infty} \lim d (p_{n}', q_{n}') =
0.\end{equation}

On the other hand, the points $p_{n}'$ and $q_{n}'$ belong to the interior of some
$2$-simplices in $U$ both isometric to their comparison triangles in $\mathbf{R}^{2}$ and which do
not coincide. So
\begin{center}$\underset{n
\rightarrow \infty} \lim d (p_{n}', q_{n}') \neq 0$.\end{center} Hence, by \eqref{3.1100}, we have
reached a
contradiction. There exists therefore a point $m$ on the path $c$ such that neither the segment
$[p,m]$ in
$U$ nor the segment $[q,m]$ in $U$ intersects the interior of $\sigma$.

\begin{figure}[h]
   \begin{center}
     \includegraphics[height=3.8cm]{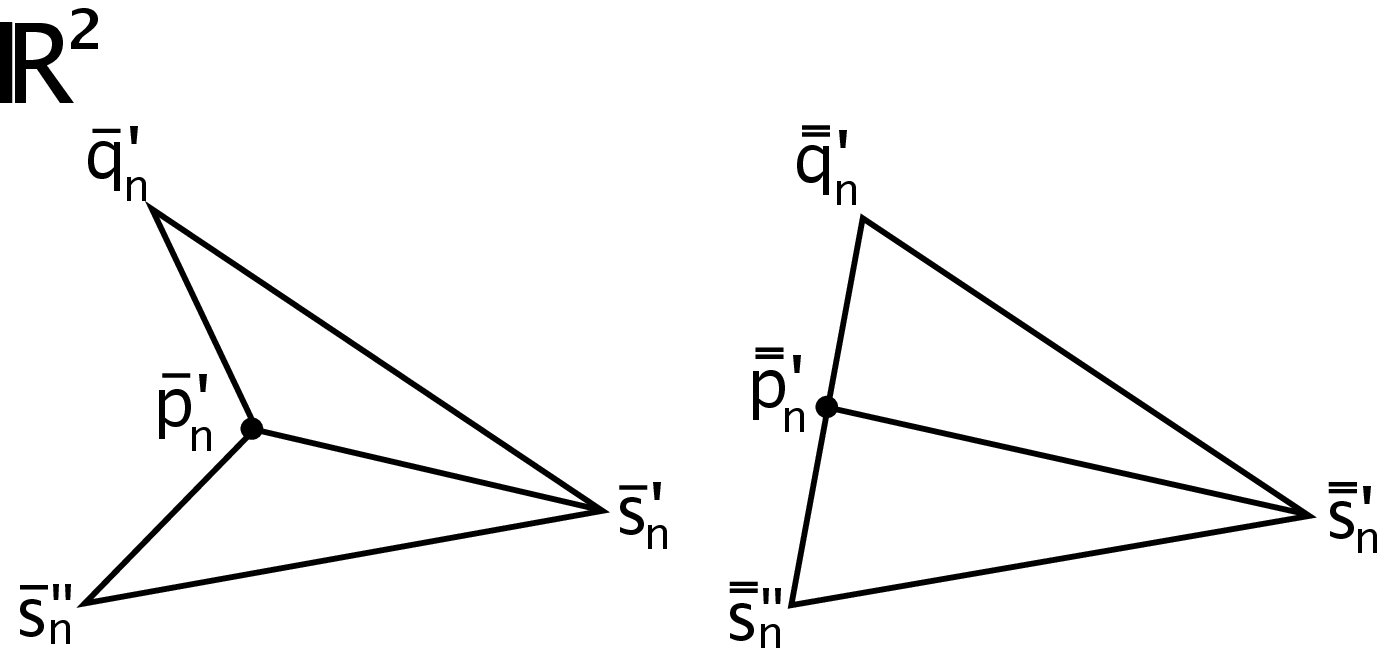}
 \caption{$Q$ is not convex\\$\overline{p}_{n}'$ lies
in the convex
hull of the other three vertices of $Q$}
 \end{center}
\end{figure}

We analyze further the case when $Q$ is not convex. Suppose $\overline{p}_{n}'$ is the vertex of $Q$
in the convex
hull of the other three vertices of $Q$ (see the figure above). The other three cases can be
handeled similarly.
Let $\triangle(\overline{\overline{s}}_{n}',\overline{\overline{q}}_{n}',\overline{\overline{s}}_{n}
'')$ be a geodesic triangle in $\mathbf{R}^{2}$ whose side lengths are equal to
$d_{\mathbf{R}^{2}}(\overline{s}_{n}', \overline{q}_{n}'), d_{\mathbf{R}^{2}}(\overline{q}_{n}',
\overline{p}_{n}') + d_{\mathbf{R}^{2}}(\overline{p}_{n}',
\overline{s}_{n}'')$ and $d_{\mathbf{R}^{2}}(\overline{s}_{n}'', \overline{s}_{n}')$,
respectively.
Let $\overline{\overline{p}}_{n}'$ be a point on
$[\overline{\overline{q}}_{n}',\overline{\overline{s}}_{n}'']$ such that
$d_{\mathbf{R}^{2}}(\overline{\overline{p}}_{n}',
\overline{\overline{q}}_{n}') = d_{\mathbf{R}^{2}}(\overline{p}_{n}',
\overline{q}_{n}')$.
Because $d(s_{n}', s_{n}'') = d_{\mathbf{R}^{2}}(\overline{s}_{n}',
\overline{s}_{n}'') =
d_{\mathbf{R}^{2}}(\overline{\overline{s}}_{n}',
\overline{\overline{s}}_{n}'')$, by
\eqref{3.1000}, we have \begin{center}$\underset{n \rightarrow \infty} \lim
d_{\mathbf{R}^{2}}(\overline{\overline{s}}_{n}',
\overline{\overline{s}}_{n}'') = 0$.\end{center}So
\begin{center}$\underset{n \rightarrow \infty}
\lim \angle_{\overline{\overline{q}}_{n}'}(\overline{\overline{s}}_{n}',
\overline{\overline{s}}_{n}'') = 0.$\end{center}
Hence \begin{equation}\label{3.1200}
\underset{n \rightarrow \infty}
\lim d_{\mathbf{R}^{2}}(\overline{\overline{s}}_{n}',
\overline{\overline{p}}_{n}') = 0.\end{equation}
Note that $\angle_{\overline{\overline{p}}_{n}'}(\overline{\overline{q}}_{n}',
\overline{\overline{s}}_{n}') +
\angle_{\overline{\overline{p}}_{n}'}(\overline{\overline{s}}_{n}',
\overline{\overline{s}}_{n}'') = \pi$, while $\angle_{\overline{p}_{n}'}(\overline{q}_{n}', \overline{s}_{n}') + \angle_{\overline{p}_{n}'}(\overline{s}_{n}', \overline{s}_{n}'') > \pi$
(this inequality holds because $Q$ is not convex and $\overline{p_{n}}'$ lies in the interior of the convex hull of the other three vertices of $Q$). Hence Alexandrov's Lemma implies
\begin{center}$d_{\mathbf{R}^{2}}(\overline{s}_{n}',
\overline{p}_{n}') \leq d_{\mathbf{R}^{2}}(\overline{\overline{s}}_{n}',
\overline{\overline{p}}_{n}')$.\end{center} Relation \eqref{3.1200} therefore ensures
\begin{center}
$\underset{n \rightarrow \infty}
\lim d_{\mathbf{R}^{2}}(\overline{s}_{n}',
\overline{p}_{n}') = 0.$ \end{center}
Thus, by \eqref{3.1400}, \begin{equation}\label{3.1300}
\underset{n \rightarrow \infty}
\lim d(s_{n}', p_{n}') = 0. \end{equation}

On the other hand, the point $p_{n}'$ lies in the interior of a $2$-dimensional face of $\sigma$
whereas the
point $s_{n}'$ lies on a path that does not intersect $\sigma$. Hence
\begin{center}$\underset{n \rightarrow \infty}
\lim d (s_{n}', p_{n}') \neq 0$\end{center} which implies, by \eqref{3.1300}, a
contradiction. There exists therefore a point $m$ on the path $c$ such that neither the segment
$[p,m]$ in $U$ nor the segment $[q,m]$ in $U$ intersects the interior of $\sigma$.

\end{proof}

\begin{lemma}\label{3.15}
Let $s$ be a point on $e_{1}$ such that $\angle_{s}(a,p_{1}) =
\angle_{s}(b,q_{1})$ and $\angle_{s}(a,q_{1}) = \angle_{s}(b,p_{1})$.
Let $c: [0,1] \rightarrow U$ be a path in $U$ joining $p$ to
$q$ that does not intersect $\sigma$. Let $m$ be a point on $c$ such that neither the segment
$[p,m]$ in $U$ nor the segment $[q,m]$ in $U$ intersects the interior of $\sigma$. Then, the
following inequality holds in $U'$: $d'(p,s) + d'(s,q) < d'(p,m) + d'(m,q)$.
\end{lemma}

\begin{proof}\label{3.16}
By Lemma \ref{3.7} and \ref{3.11}, such points $s$ and $m$ exist. Moreover, such point $s$ is
unique. Let $l$ be some point on $e_{1}$.

\begin{figure}[h]
   \begin{center}
     \includegraphics[height=5cm]{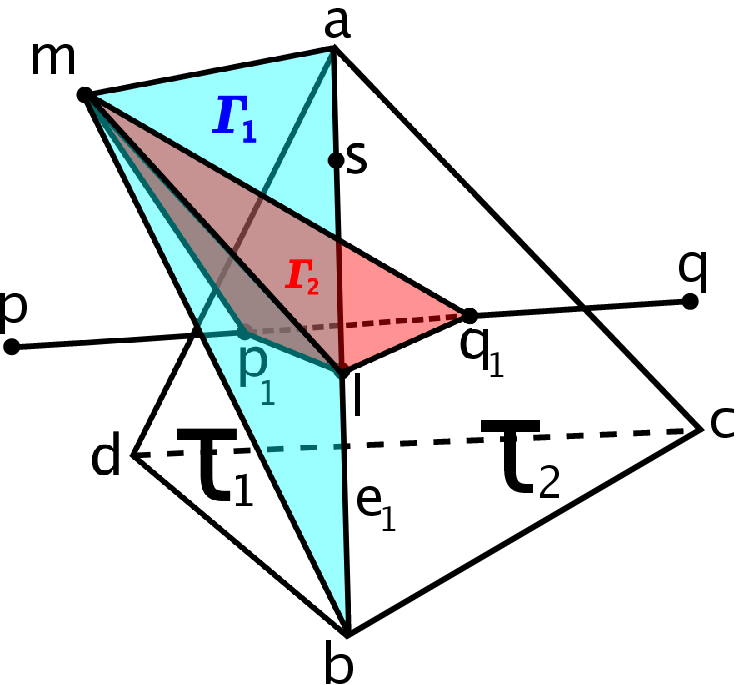}
 \caption{$\varGamma_{1}$ and $\varGamma_{2}$}
 \end{center}
\end{figure}
We denote by $\varGamma_{1}$ the union of the geodesic triangles $\triangle(a,m,l)$ and
$\triangle(b,m,l)$. Note that $\varGamma_{1}$ intersects the boundary of $\sigma$ along the edge
$e_{1} = [a,b]$, i.e. along one common boundary edge $\tau_{1}$ and $\tau_{2}$ of
(which are two $2$-dimensional faces of $\sigma$).
Further, note that
\begin{equation}\label{3.1700}\angle_{l}(a,m) + \angle_{l}(m,b) = \pi.\end{equation}

We denote by $\varGamma_{2}$ the union of the geodesic triangles $\triangle(p_{1},m,l)$ and
$\triangle(q_{1},m,l)$. Note that $\varGamma_{2}$ intersects the boundary of $\sigma$ along two
sides of the geodesic
triangle $\triangle(p_{1},q_{1},l)$, i.e. along the interior of $\tau_{1}$ and $\tau_{2}$ which are also two $2$-dimensional faces of
$\sigma$. Hence relation
\eqref{3.1700} guarantees that \begin{center}$\angle_{l}(p_{1},m) +
\angle_{l}(m,q_{1}) > \pi$.\end{center}
Because $l \in e_{1}$ while $m \in |U|, m \notin e_{1}$,
Corollary A ensures that \begin{center}$\angle_{l}(p,m)
+
\angle_{l}(m,q) > \pi$.\end{center}
Thus, by Lemma \ref{3.41}, we have \begin{equation}\label{3.2400} d(p,l) + d(l,q) < d(p,m) +
d(m,q).\end{equation}According to Lemma \ref{3.9}, it follows that
\begin{equation}\label{3.2555}d(p,s) +
d(s,q) < d(p,t) +
d(t,q),\end{equation} for any $t$ on
$e_{1}$ that differs from $s$. The above relations ensure that
\begin{center}$d(p,s) + d(s,q) < d(p,m) + d(m,q)$.\end{center}
Because the segments $[p,s], [s,q],
[p,m]$ and $[m,q]$ in $U$
do not intersect the interior of $\sigma$, the same inequality holds in $U'$:
\begin{center}$d'(p,s)
+ d'(s,q) < d'(p,m) + d'(m,q).$\end{center}

\end{proof}

We find further the geodesic segments in $U'$ joining those pairs of points that are joined in $U$
by a segment that intersects the interior of $\sigma$. Because $K$ satisfies Property A, there are no geodesic segments $[p,q]$ in $U$ such that the points $p,q$
are joined in $U'$ by two geodesic segments $\gamma_{1}, \gamma_{2}$ of equal length such that $\gamma_{1}$ intersects one, while $\gamma_{2}$
intersects one or two
of the boundary edges of
$\alpha$ (if $\alpha$ is $2$-dimensional) or from $\alpha$ itself (if $\alpha$ is $1$-dimensional).

\begin{lemma}\label{3.17}
Let $s$ be a point on $e_{1}$ such that $\angle_{s}(a,p_{1}) =
\angle_{s}(b,q_{1})$ and $\angle_{s}(a,q_{1}) = \angle_{s}(b,p_{1})$.
Let $t$ be a point on $e_{2}$ such that $\angle_{t}(c,d) =
\angle_{t}(a,p_{1})$ and $\angle_{t}(a,c) = \angle_{t}(d,p_{1})$.
Let $v$ be a point on $e_{3}$ such that $\angle_{v}(c,q_{1}) =
\angle_{v}(a,d)$ and $\angle_{v}(a,q_{1}) = \angle_{v}(c,d)$.
If $d'(p,s) + d'(s,q) \leq d'(p,t) + d'(t,v) + d'(v,q)$, then the geodesic
segment $[p,q]$ in $U'$ with respect to $d'$ is the union of the geodesic segments $[p,s]$ and
$[s,q]$. Otherwise, the geodesic segment $[p,q]$ in $U'$ with respect to $d'$ is the union of
the geodesic segments $[p,t], [t,v]$ and $[v,q]$.

\end{lemma}

\begin{proof}\label{3.18}

Because $U$ is a CAT(0) space, Lemma \ref{3.7} guarantees that the points $s, t$ and $v$ exist and
they are unique.

\begin{figure}[h]
   \begin{center}
     \includegraphics[height=4cm]{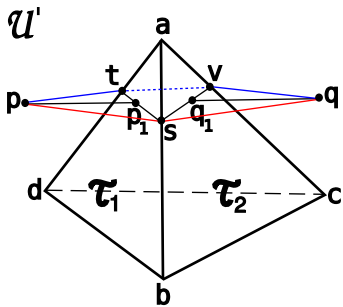}
        \caption{The geodesic segment $[p,q]$ in $U'$}
 \end{center}
\end{figure}

In case $d'(p,s) + d'(s,q) \leq
d'(p,t) + d'(t,v) + d'(v,q)$, let $c : [0,1] \rightarrow U'$ denote the path
obtained by concatenating the segments $[p,s]$ and $[s,q]$. Among all paths
joining $p$ to $q$ in $U'$ which pass through $s$, the path $c$ has
the shortest length.

Suppose that there exists a path $c_{0} : [0,1] \rightarrow U'$
connecting $p$ to $q$ in $U'$ that does not pass through $s$ and
whose length is less or equal to the length of the path $c$. Because
the path $c_{0}$ does not intersect $\sigma$, there exists,
according to Lemma \ref{3.11}, a point $m$ on $c_{0}$  such that
the geodesic segments $[p,m]$ and $[m,q]$ in $U$ do not intersect the interior of
$\sigma$. The geodesic segments $[p,m]$ and $[m,q]$ in $U$
belong therefore to $U'$. So

\begin{center}
$d'(p,m) + d'(m,q) \leq l(c_{0}) \leq l(c) = d'(p,s) + d'(s,q)$\\
\end{center}
which is, by Lemma \ref{3.15}, a contradiction. Any path in $U'$
joining $p$ to $q$ and which does not pass through $s$, is therefore longer
than $c$.

Altogether it follows that the geodesic segment joining $p$ to $q$
in $U'$ with respect to $d'$ is the union of the geodesic segments
$[p,s]$ and $[s,q]$.

In case $d'(p,s) + d'(s,q) >
d'(p,t) + d'(t,v) + d'(v,q)$, one can similarly show, applying Lemma \ref{3.15}
twice, that the geodesic segment joining $p$ to $q$
in $U'$ with respect to $d'$ is the union of the geodesic segments
$[p,t], [t,v]$ and $[v,q]$. Namely, Lemma \ref{3.15} will ensure that the geodesic segment joining $p$
to $v$ ($p$
to $q$) in $U'$ with respect to $d'$ is the union of the geodesic segments $[p,t]$ and $[t,v]$
($[p,v]$ and $[v,q]$).

\end{proof}

Using the CAT(0) inequality and Alexandrov's Lemma, we show further that any geodesic triangle in
$U'$ satisfies the CAT(0)
inequality. Depending on the position of the vertices of such geodesic triangle with respect to
the $3$-simplex of $U$ with the free face, we must consider eight cases.
We will find geodesic segments
in $U'$ using, mostly without stating so explicitely, Lemma \ref{3.17}.

\begin{lemma}\label{3.19}
Let $r$ be a point in $U$ such that the geodesic segments  $[r,p]$
and $[r,q]$  do not intersect the interior of $\sigma$. Let $s$ be a point on $e_{1}$ such that
$\angle_{s}(a,p_{1}) = \angle_{s}(b,q_{1})$ and $\angle_{s}(a,q_{1}) = \angle_{s}(b,p_{1})$. Let $t$
be a point on $e_{2}$ such that $\angle_{t}(c,d) = \angle_{t}(a,p_{1})$ and $\angle_{t}(a,c) =
\angle_{t}(d,p_{1})$. Let $v$ be a point on $e_{3}$ such that $\angle_{v}(c,q_{1}) =
\angle_{v}(a,d)$ and $\angle_{v}(a,q_{1}) = \angle_{v}(c,d)$. If $d'(p,s) + d'(s,q) < d'(p,t) +
d'(t,v) + d'(v,q)$, then the geodesic triangle $\triangle (p,q,r)$ in $U'$ satisfies the CAT(0)
inequality.
\end{lemma}

\begin{proof}\label{3.20}

Because $U$ is a CAT(0) space, according to Lemma \ref{3.7}, the point $s,t$ and $v$ exist and they
are
unique. Lemma \ref{3.17} ensures that $d'(p,q) = d'(p,s) + d'(s,q)$.

\begin{figure}[h]
   \begin{center}
     \includegraphics[height=4cm]{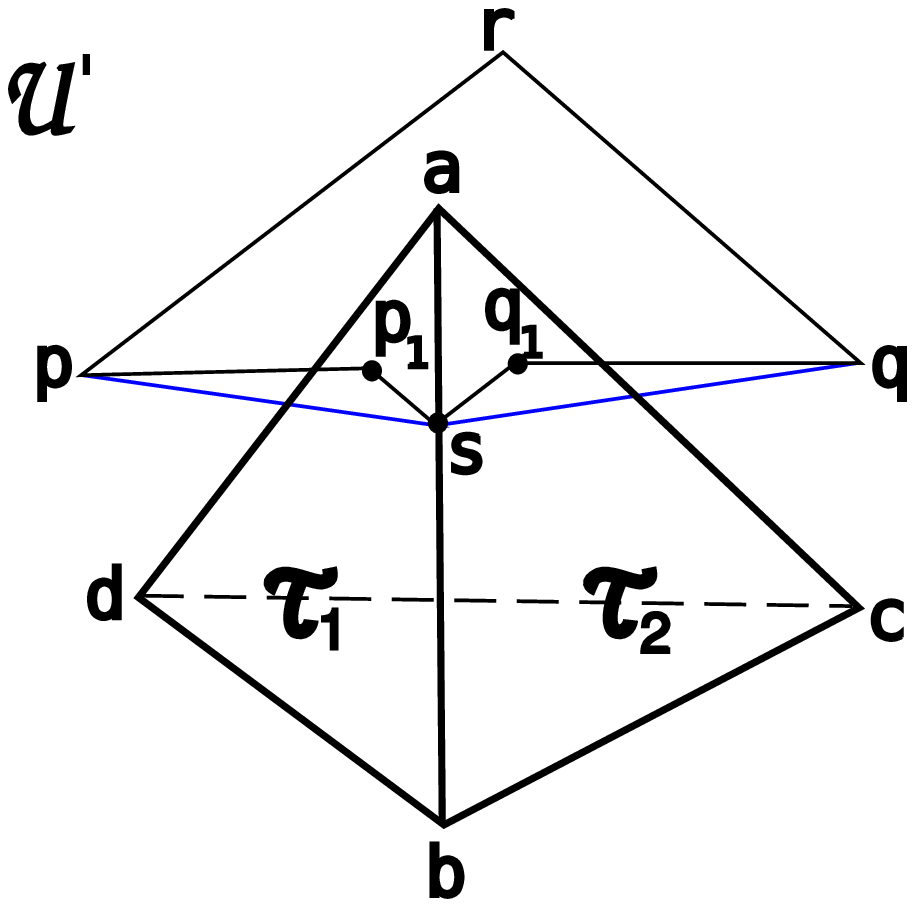}
        \caption{The geodesic triangle $\triangle(p,q,r)$ in $U'$ satisfies the CAT(0) inequality}
 \end{center}
\end{figure}

Let $\triangle (p',q',r')$ be a comparison triangle in
$\mathbf{R}^{2}$ for the geodesic triangle
$\triangle (p,q,r)$ in $U'$. Let $s' \in [p',q']$ be a comparison point for $s \in [p,q]$.

Let $\triangle (p'',r'',s'')$ be a comparison triangle in
$\mathbf{R}^{2}$ for the geodesic triangle $\triangle (p,r,s)$ in $U$ and let $\triangle
(r'',s'',q'')$ be a comparison triangle in $\mathbf{R}^{2}$ for
the geodesic triangle $\triangle (r,s,q)$ in $U$. We place the comparison triangles
$\triangle (p'',r'',s'')$ and $\triangle (r'',s'',q'')$ in
different half-planes with respect to the line $r''s''$ in
$\mathbf{R}^{2}$.

\begin{figure}[h]
   \begin{center}
     \includegraphics[height=3cm]{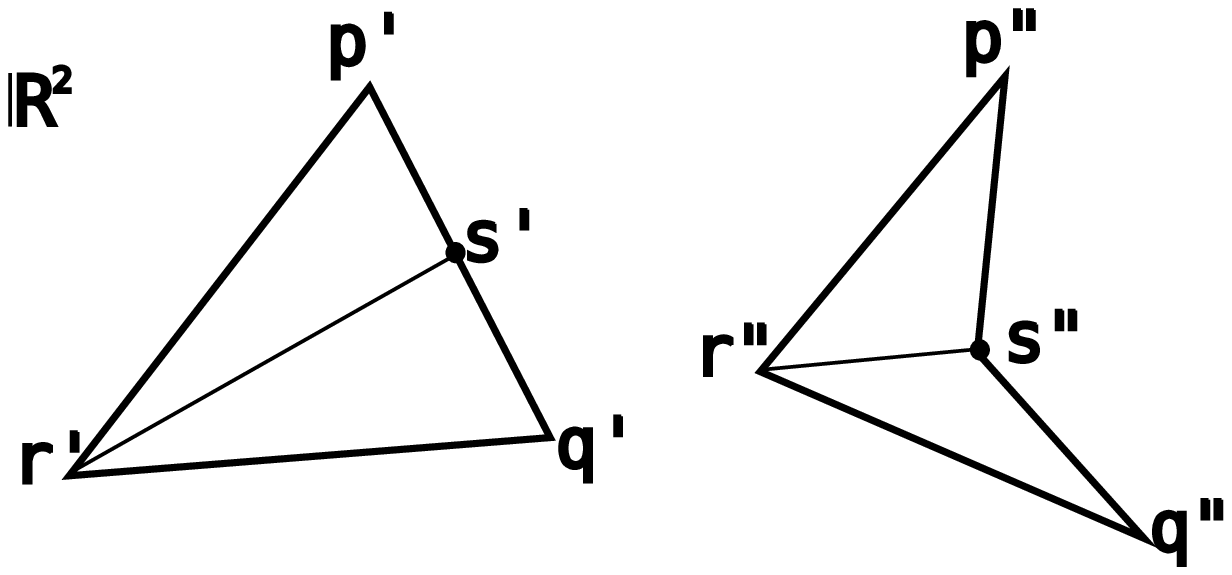}
        \caption{Comparison triangles in $\mathbf{R}^{2}$}
 \end{center}
\end{figure}

By the CAT(0) inequality, we have $\angle_{r}(p,s) \leq
\angle_{r''}(p'',s'')$, $\angle_{r}(s,q) \leq
\angle_{r''}(s'',q'')$, $\angle_{p}(r,s) \leq
\angle_{p''}(r'',s'')$ and $\angle_{q}(r,s) \leq
\angle_{q''}(r'',s'')$. Because $\angle_{s'}(p',r') + \angle_{s'}(r',q') = \pi$, Alexandrov's
Lemma further implies: $\angle_{r''}(p'',s'') \leq
\angle_{r'}(p',s')$, $\angle_{r''}(s'',q'') \leq
\angle_{r'}(s',q')$, $\angle_{p''}(r'',s'') \leq
\angle_{p'}(r',s')$ and $\angle_{q''}(r'',s'') \leq
\angle_{q'}(r',s')$. Altogether it follows that $\angle_{r}(p,q) \leq \angle_{r}(p,s) +
\angle_{r}(s,q) \leq \angle_{r'}(p',s') + \angle_{r'}(s',q') = \angle_{r'}(p',q')$,
$\angle_{p}(r,s) \leq \angle_{p'}(r',s')$ and $\angle_{q}(r,s) \leq
\angle_{q'}(r',s')$. So the geodesic triangle $\triangle (p,q,r)$ in $U'$ satisfies the CAT(0)
inequality.

\end{proof}

\begin{lemma}\label{3.21}

Let $r$ be a point in $U$ such that the geodesic segments  $[r,p]$
and $[r,q]$ do not intersect the interior of $\sigma$.
Let $s$ be a point on $e_{1}$ such that $\angle_{s}(a,p_{1})
= \angle_{s}(b,q_{1})$ and $\angle_{s}(a,q_{1}) = \angle_{s}(b,p_{1})$. Let $t$ be a point on
$e_{2}$ such that $\angle_{t}(c,d) = \angle_{t}(a,p_{1})$ and $\angle_{t}(a,c) =
\angle_{t}(d,p_{1})$. Let $v$ be a point on $e_{3}$ such that $\angle_{v}(c,q_{1}) =
\angle_{v}(a,d)$ and $\angle_{v}(a,q_{1}) = \angle_{v}(c,d)$. If $d'(p,t) +
d'(t,v) + d'(v,q) < d'(p,s) + d'(s,q)$, then the geodesic triangle $\triangle (p,q,r)$ in $U'$
satisfies the CAT(0) inequality.
\end{lemma}

\begin{proof}\label{3.22}

Because $U$ is a CAT(0) space, by Lemma \ref{3.7}, the points $s,t$ and $v$ exist and they are
unique. Lemma \ref{3.17} further implies that $d'(p,q) = d'(p,t) + d'(t,v) + d'(v,q)$.

\begin{figure}[h]
   \begin{center}
     \includegraphics[height=4cm]{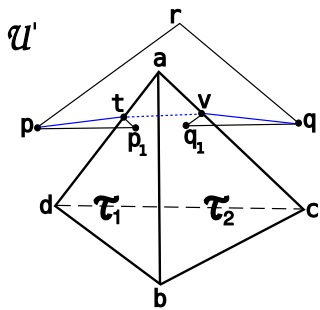}
        \caption{The geodesic triangle $\triangle(p,q,r)$ in $U'$ satisfies the CAT(0) inequality}
 \end{center}
\end{figure}

Let $\triangle (p',q',r')$ be a comparison triangle in $\mathbf{R}^{2}$ for the geodesic triangle
$\triangle (p,q,r)$ in $U'$. Let $t' \in [p',q']$ be a comparison point for $t \in [p,q]$
and let $v' \in [p',q']$ be a comparison point for $v \in [p,q]$.

Let $\triangle (p'',r'',t'')$ be a comparison triangle in $\mathbf{R}^{2}$ for the geodesic
triangle $\triangle (p,r,t)$ in $U$. Let $\triangle (r'',t'',v'')$ be a comparison triangle in
$\mathbf{R}^{2}$ for the geodesic triangle $\triangle (r,t,v)$ in $U$. We place the geodesic
triangles $\triangle (p'',r'',t'')$ and $\triangle (r'',t'',v'')$ in different half-planes with
respect to the line $r''t''$ in $\mathbf{R}^{2}$. By the CAT(0) inequality,
$\angle_{r}(p,t) \leq \angle_{r''}(p'',t''), \angle_{r}(t,v) \leq
\angle_{r''}(t'',v'')$ and $\angle_{p}(r,t) \leq \angle_{p''}(r'',t'').$

Let $\triangle (p''',r''',v''')$ be a comparison triangle in
$\mathbf{R}^{2}$ for the geodesic triangle $\triangle (p,r,v)$ in $U$. Let $t''' \in [p''',v''']$
be a comparison point for $t \in [p,v]$.
Let $\triangle
(r''',v''',q''')$ be a comparison triangle in $\mathbf{R}^{2}$ for
the geodesic triangle $\triangle (r,v,q)$ in $U$. We place the comparison triangles
$\triangle (p''',r''',v''')$ and $\triangle (r''',v''',q''')$ in
different half-planes with respect to the line $r'''v'''$ in
$\mathbf{R}^{2}$. By the CAT(0) inequality,
$\angle_{q}(r,v) \leq \angle_{q'''}(r''',v''')$ and $\angle_{r}(v,q) \leq \angle_{r'''}(v''',q''').$

Because $\angle_{t'''}(p''',r''') + \angle_{t'''}(r''',v''') = \pi$ and $\angle_{v'}(p',r') +
\angle_{v'}(r',q') = \pi$, Alexandrov's Lemma guarantees that $\angle_{p''}(r'',t'') \leq
\angle_{p'''}(r''',t''') \leq
\angle_{p'}(r',t'), \angle_{r''}(p'',v'') \leq
\angle_{r'''}(p''',v''') \leq
\angle_{r'}(p',v'),$ while $\angle_{q'''}(r''',v''') \leq \angle_{q'}(r',v')$. Altogether it
follows that $\angle_{p}(r,q) \leq
\angle_{p'}(r',q'), \angle_{q}(p,r) \leq
\angle_{q'}(p',r')$ and $\angle_{r}(p,q) \leq
\angle_{r'}(p',q').$ So the geodesic triangle $\triangle(p,q,r)$ in $U'$ satisfies the
CAT(0) inequality.

\end{proof}

\begin{lemma}\label{3.23}
Let $r$ be a point in $U$ such that the geodesic segment $[r,q]$ does not intersect the
interior of $\sigma$ whereas the geodesic segment $[p,r]$ intersects the interior
of $\tau_{1}$ in $p_{2}$, and the interior
of $\tau_{2}$ in $r_{1}$. Let $s$ be a point on $e_{1}$ such that
$\angle_{s}(a,p_{1}) = \angle_{s}(b,q_{1})$ and $\angle_{s}(a,q_{1}) = \angle_{s}(b,p_{1})$. Let
$t$ be a point on $e_{1}$ such that $\angle_{t}(a,p_{2}) =
\angle_{t}(b,r_{1})$ and $\angle_{t}(a,r_{1}) = \angle_{t}(b,p_{2})$. If $d'(p,q) =
d'(p,s) + d'(s,q)$ and $d'(p,r) = d'(p,t) + d'(t,r)$, then the geodesic triangle $\triangle
(p,q,r)$ in $U'$ satisfies the CAT(0) inequality.
\end{lemma}

\begin{proof}\label{3.24}

Because $U$ is a CAT(0) space, Lemma \ref{3.7} implies that the points $s$ and $t$ exist and they
are unique.

\begin{figure}[h]
   \begin{center}
     \includegraphics[height=4cm]{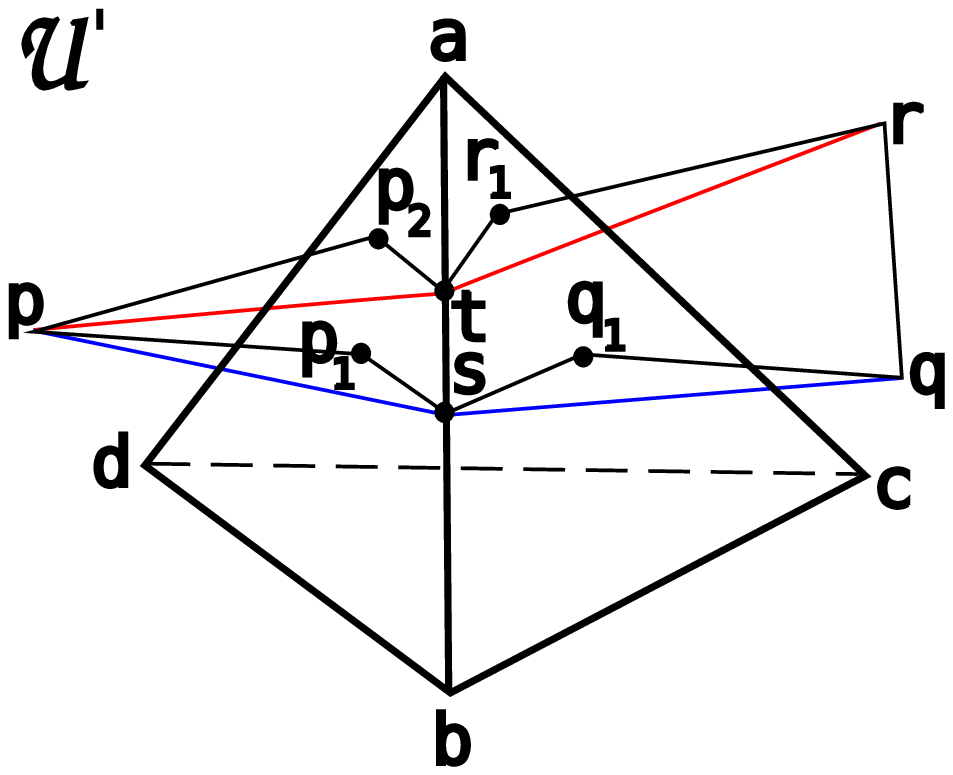}
         \caption{The geodesic triangle $\triangle(p,q,r)$ in $U'$ satisfies the CAT(0) inequality}
 \end{center}
\end{figure}

Let $\triangle (p',q',r')$ be a comparison triangle in $\mathbf{R}^{2}$ for the
geodesic triangle $\triangle (p,q,r)$ in $U'$. Let $s' \in [p',q']$ be a comparison point for $s
\in
[p,q]$. Let $t' \in [p',r']$ be a comparison point for $t \in [p,r]$.

Let $\triangle (p'',t'',s'')$ be a comparison triangle in
$\mathbf{R}^{2}$ for the geodesic triangle $\triangle (p,t,s)$ in $U$. Let $\triangle
(t'',s'',r'')$ be a comparison triangle in $\mathbf{R}^{2}$ for
the geodesic triangle $\triangle (t,s,r)$ in $U$. We place the comparison triangles
$\triangle (p'',t'',s'')$ and $\triangle (t'',s'',r'')$ in
different half-planes with respect to the line $t''s''$ in
$\mathbf{R}^{2}$. The CAT(0) inequality implies that $\angle_{p}(t,s) \leq
\angle_{p''}(t'',s'')$, $\angle_{r}(t,s) \leq
\angle_{r''}(t'',s'')$.

Let $\triangle (p''',r''',s''')$ be a comparison triangle in
$\mathbf{R}^{2}$ for the geodesic triangle $\triangle (p,r,s)$ in $U$. Let $t''' \in [p''',r''']$
be a comparison point for $t \in [p,r]$. Let $\triangle (q''',r''',s''')$ be a comparison triangle
in $\mathbf{R}^{2}$ for the geodesic triangle $\triangle (q,r,s)$ in $U$. We place the comparison
triangles $\triangle (p''',r''',s''')$ and $\triangle (q''',r''',s''')$ in
different half-planes with respect to the line $r'''s'''$ in
$\mathbf{R}^{2}$.

Because $\angle_{t'''}(p''',s''') +
\angle_{t'''}(s''',r''') = \pi$, by Alexandrov's Lemma we have $\angle_{p''}(t'',s'') \leq
\angle_{p'''}(t''',s''')$, $\angle_{r''}(t'',s'') \leq \angle_{r'''}(t''',s''')$.

Note that $\angle_{s'}(p',r') +
\angle_{s'}(r',q') = \pi$. So, by Alexandrov's Lemma and the CAT(0)
inequality, we have $\angle_{q}(r,s) \leq \angle_{q'''}(r''',s''') \leq \angle_{q'}(r',s')$,
$\angle_{r}(p,q) \leq \angle_{r}(p,s) + \angle_{r}(s,q) \leq
\angle_{r'''}(p''',s''') + \angle_{r'''}(s''',q''') \leq \angle_{r'}(p',s') + \angle_{r'}(s',q') =
\angle_{r'}(p',q'), \angle_{p}(r,s) \leq \angle_{p'''}(r''',s''') \leq \angle_{p'}(r',s')$.
Thus the geodesic triangle $\triangle (p,q,r)$ in $U'$ satisfies the CAT(0)
inequality.

\end{proof}

\begin{lemma}\label{3.25}
Let $r$ be a point in $U$ such that the geodesic segment $[r,q]$ does not intersect the
interior of $\sigma$ whereas the geodesic segment $[p,r]$ intersects the interior of
$\tau_{1}$ in $p_{2}$, and the interior of $\tau_{3}$ in
$r_{1}$. Let $s$ be a point on $e_{1}$ such that $\angle_{s}(a,p_{1}) =
\angle_{s}(b,q_{1})$ and $\angle_{s}(a,q_{1}) = \angle_{s}(b,p_{1})$. Let
$t$ be a point on $e_{2}$ such that $\angle_{t}(a,p_{2}) =
\angle_{t}(d,r_{1})$ and $\angle_{t}(a,r_{1}) = \angle_{t}(d,p_{2})$. If $d'(p,q) =
d'(p,s) + d'(s,q)$ and $d'(p,r) = d'(p,t) + d'(t,r)$, then the geodesic triangle $\triangle
(p,q,r)$ in $U'$ satisfies the CAT(0) inequality.
\end{lemma}

\begin{proof}\label{3.26}

Because $U$ is a CAT(0) space, by Lemma \ref{3.7}, the points $s$ and $t$ exist and they are
unique.

\begin{figure}[h]
   \begin{center}
     \includegraphics[height=4cm]{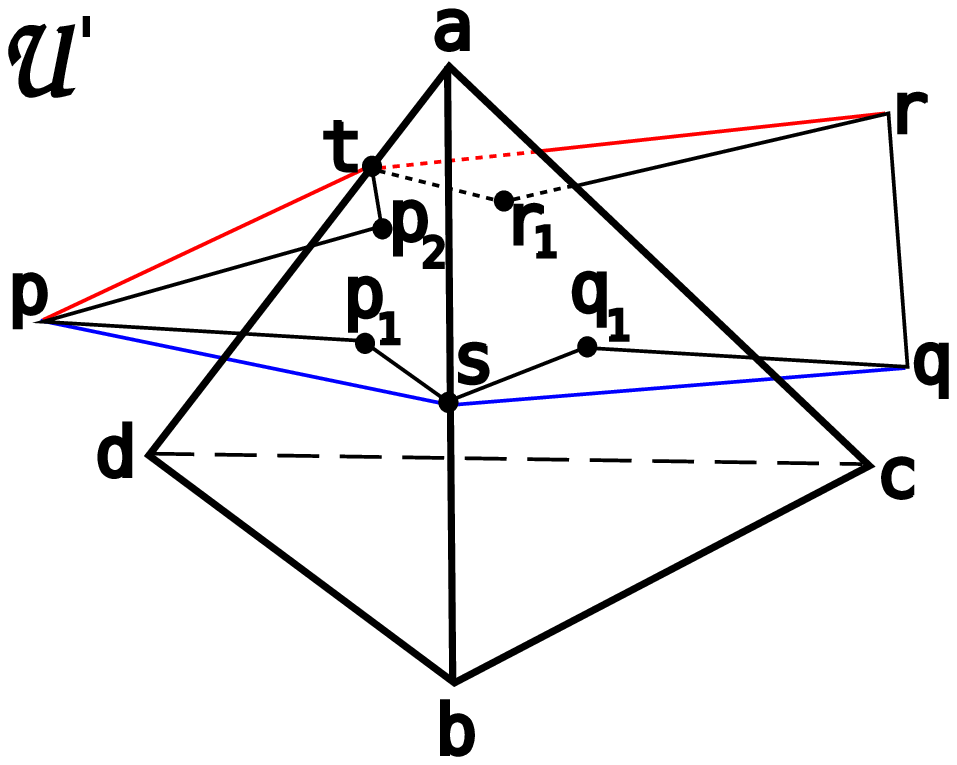}
         \caption{The geodesic triangle $\triangle(p,q,r)$ in $U'$ satisfies the CAT(0) inequality}
\end{center}
\end{figure}

Let $\triangle (p',q',r')$ be a comparison triangle in $\mathbf{R}^{2}$ for the
geodesic triangle $\triangle (p,q,r)$ in $U'$. Let $s' \in [p',q']$ be a comparison point for $s
\in [p,q]$. Let $t' \in [p',r']$ be a comparison point for $t \in [p,r]$.

Let $\triangle (p'',t'',s'')$ be a comparison triangle in
$\mathbf{R}^{2}$ for the geodesic triangle $\triangle (p,t,s)$ in $U$. Let $\triangle
(t'',s'',r'')$ be a comparison triangle in $\mathbf{R}^{2}$ for the geodesic triangle
$\triangle (t,s,r)$ in $U$. We place the comparison triangles
$\triangle (p'',t'',s'')$ and $\triangle (t'',s'',r'')$ in
different half-planes with respect to the line $t''s''$ in
$\mathbf{R}^{2}$. By the CAT(0) inequality we have $\angle_{p}(t,s) \leq
\angle_{p''}(t'',s'')$, $\angle_{r}(t,s) \leq
\angle_{r''}(t'',s'')$.

Let $\triangle (p''',r''',s''')$ be a comparison triangle in
$\mathbf{R}^{2}$ for the geodesic triangle $\triangle (p,r,s)$ in $U$. Let $t''' \in [p''',r''']$
be a comparison point for $t \in [p,r]$.
Let $\triangle (q''',r''',s''')$ be a comparison triangle in
$\mathbf{R}^{2}$ for the geodesic triangle $\triangle (q,r,s)$ in $U$. We place the comparison
triangles $\triangle (p''',r''',s''')$ and $\triangle (q''',r''',s''')$ in
different half-planes with respect to the line $r'''s'''$ in
$\mathbf{R}^{2}$.

Because $\angle_{t'''}(p''',s''') +
\angle_{t'''}(s''',r''') = \pi$, Alexandrov's Lemma
guarantees that $\angle_{p''}(t'',s'') \leq
\angle_{p'''}(t''',s''')$, $\angle_{r''}(t'',s'') \leq
\angle_{r'''}(t''',s''')$.
The CAT(0) inequality ensures that $\angle_{q}(r,s) \leq
\angle_{q'''}(r''',s''')$ and $\angle_{r}(s,q) \leq
\angle_{r'''}(s''',q''')$.

Because $\angle_{s'}(p',r') +
\angle_{s'}(r',q') = \pi$, Alexandrov's Lemma implies $\angle_{p'''}(r''',s''') \leq
\angle_{p'}(r',s'), \angle_{r'''}(p''',s''') \leq \angle_{r'}(p',s'), \angle_{r'''}(s''',q''')
\leq
\angle_{r'}(s',q'), \angle_{q'''}(r''',s''') \leq \angle_{q'}(r',s')$.

Altogether it follows that $\angle_{r}(p,q) \leq \angle_{r}(p,s) +
\angle_{r}(s,q) \leq \angle_{r'}(p',s') +
\angle_{r'}(s',q') = \angle_{r'}(p',q')$, $\angle_{p}(r,s) \leq
\angle_{p'}(r',s')$ and $\angle_{q}(r,s) \leq
\angle_{q'}(r',s')$. So the geodesic triangle $\triangle (p,q,r)$ in $U'$ satisfies the CAT(0)
inequality.

\end{proof}

\begin{lemma}\label{3.27}
Let $r$ be a point in $U$ such that the geodesic segment $[r,q]$ does not intersect the
interior of $\sigma$ whereas the geodesic segment $[p,r]$ intersects the interior of
$\tau_{1}$ in $p_{2}$, and the interior of $\tau_{2}$ in
$r_{1}$. Let $s$ be a point on $e_{1}$ such that $\angle_{s}(a,p_{1}) =
\angle_{s}(b,q_{1})$ and $\angle_{s}(a,q_{1}) = \angle_{s}(b,p_{1})$. Let
$t$ be a point on $e_{2}$ such that $\angle_{t}(d,p_{2}) =
\angle_{t}(a,c)$ and $\angle_{t}(a,p_{2}) = \angle_{t}(d,c)$. Let
$v$ be a point on $e_{3}$ such that $\angle_{v}(c,d) =
\angle_{v}(a,r_{1})$ and $\angle_{v}(a,d) = \angle_{v}(c,r_{1})$. If $d'(p,q) =
d'(p,s) + d'(s,q)$ and $d'(p,r) = d'(p,t) + d'(t,v) + d'(v,r)$, then the geodesic triangle
$\triangle (p,q,r)$ in $U'$ satisfies the CAT(0) inequality.
\end{lemma}

\begin{proof}\label{3.28}

Because $U$ is a CAT(0) space, by Lemma \ref{3.7}, the points $s,t$ and $v$ exist and they are
unique.

\begin{figure}[h]
   \begin{center}
     \includegraphics[height=4cm]{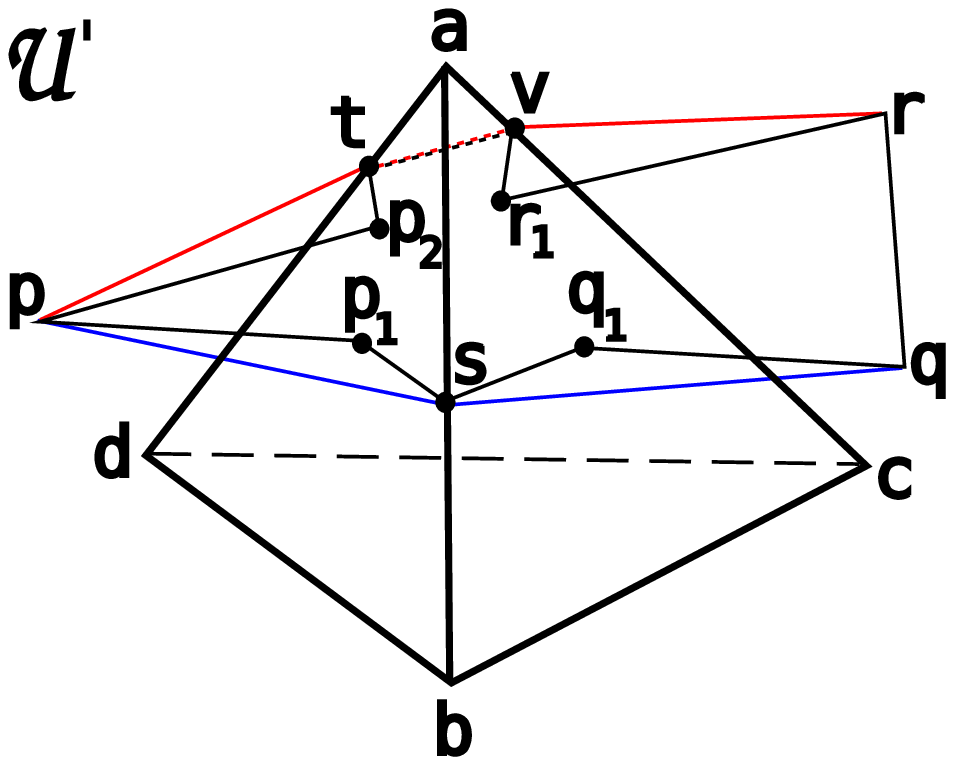}
         \caption{The geodesic triangle $\triangle(p,q,r)$ in $U'$ satisfies the CAT(0) inequality}
 \end{center}
\end{figure}

Let $\triangle (p',q',r')$ be a comparison triangle in $\mathbf{R}^{2}$ for the geodesic triangle
$\triangle (p,q,r)$ in $U'$. Let $v' \in [p',r']$ be a comparison point for $v \in [p,r]$.

Let $\triangle (p'',t'',s'')$ be a comparison triangle in
$\mathbf{R}^{2}$ for the geodesic triangle $\triangle (p,t,s)$ in $U$. Let $\triangle
(t'',s'',v'')$ be a comparison triangle in $\mathbf{R}^{2}$ for the geodesic triangle
$\triangle (t,s,v)$ in $U$. We place the comparison triangles
$\triangle (p'',t'',s'')$ and $\triangle (t'',s'',v'')$ in
different half-planes with respect to the line $t''s''$ in
$\mathbf{R}^{2}$. The CAT(0) inequality ensures that $\angle_{p}(t,s) \leq
\angle_{p''}(t'',s'')$.

Let $\triangle (p''',v''',s''')$ be a comparison triangle in
$\mathbf{R}^{2}$ for the geodesic triangle $\triangle (p,v,s)$ in $U$. Let $\triangle
(q''',v''',s''')$ be a comparison triangle in
$\mathbf{R}^{2}$ for geodesic triangle $\triangle (q,v,s)$ in $U$. We place the comparison
triangles $\triangle (p''',v''',s''')$ and $\triangle (q''',v''',s''')$ in
different half-planes with respect to the line $v'''s'''$ in
$\mathbf{R}^{2}$. Let $t''' \in [p''',v''']$ such that $d_{\mathbf{R}^{2}}(p''',t''') = d(p,t)$.

Because $\angle_{t'''}(p''',s''') +
\angle_{t'''}(s''',v''') = \pi$, Alexandrov's Lemma implies that $\angle_{p''}(t'',s'') \leq
\angle_{p'''}(t''',s''')$.

Let $\triangle (p'^{V},v'^{V},q'^{V})$ be a comparison triangle in
$\mathbf{R}^{2}$ for the geodesic triangle $\triangle (p,v,q)$ in $U$. Let $s'^{V} \in [p'^{V},
q'^{V}]$ be a comparison triangle in $\mathbf{R}^{2}$ for $s \in [p,q]$. Let $\triangle
(v'^{V},q'^{V},r'^{V})$ be a
comparison triangle in $\mathbf{R}^{2}$ for the geodesic triangle $\triangle (v,q,r)$ in $U$. We
place the comparison triangles $\triangle (p'^{V},v'^{V},q'^{V})$ and
$\triangle (v'^{V},q'^{V},r'^{V})$ in
different half-planes with respect to the line $v'^{V}q'^{V}$ in
$\mathbf{R}^{2}$.

Because $\angle_{s'^{V}}(p'^{V}, v'^{V}) + \angle_{s'^{V}}(v'^{V},
q'^{V}) = \pi$ and $\angle_{v'}(p',q') + \angle_{v'}(q',r') = \pi$, Alexandrov's Lemma ensures
that $\angle_{p'''}(v''', s''') \leq
\angle_{p'^{V}}(v'^{V}, s'^{V})$ and $\angle_{p'^{V}}(v'^{V},q'^{V}) \leq \angle_{p'}(v',q')$.

Altogether it follows that in $U'$ we have $\angle_{p}(t,s) = \angle_{p}(r,q) \leq
\angle_{p'}(r',q')$. One can similarly show that $\angle_{r}(p,q) \leq
\angle_{r'}(p',q')$, $\angle_{q}(p,r) \leq
\angle_{q'}(p',r')$. So the geodesic triangle $\triangle (p,q,r)$ in $U'$ satisfies the CAT(0)
inequality.

\end{proof}

\begin{lemma}\label{3.29}
Let $r$ be a point in $U$ such that the geodesic segment $[r,q]$ does not intersect the interior
of
$\sigma$ whereas the geodesic segment $[p,r]$ intersects the interior of $\tau_{1}$ in
$p_{2}$, and the interior of $\tau_{3}$ in $r_{1}$. Let $s$ be a point on $e_{1}$ such that
$\angle_{s}(a,p_{1}) = \angle_{s}(b,q_{1})$ and $\angle_{s}(a,q_{1}) = \angle_{s}(b,p_{1})$. Let
$t$ be a point on $e_{1}$ such that $\angle_{t}(a,p_{2}) =
\angle_{t}(b,c)$ and $\angle_{t}(a,c) = \angle_{t}(p_{2},b)$. Let
$v$ be a point on $e_{3}$ such that $\angle_{v}(b,c) =
\angle_{v}(a,r_{1})$ and $\angle_{v}(b,a) = \angle_{v}(c,r_{1})$. If $d'(p,q) =
d'(p,s) + d'(s,q)$ and $d'(p,r) = d'(p,t) + d'(t,v) + d'(v,r)$, then the geodesic triangle
$\triangle (p,q,r)$ in $U'$ satisfies the CAT(0) inequality.
\end{lemma}

\begin{proof}\label{3.30}

Because $U$ is a CAT(0) space, by Lemma \ref{3.7}, the points $s,t$ and $v$ exist and they
are unique.

Let $\triangle (p',q',r')$ be a comparison triangle in $\mathbf{R}^{2}$ for the geodesic triangle
$\triangle (p,q,r)$ in $U'$. Let $s' \in [p',q']$ be a comparison point for $s \in [p,q]$ and let
$v' \in [p',r']$ be a comparison point for $v \in [p,r]$.

\begin{figure}[h]
   \begin{center}
     \includegraphics[height=4cm]{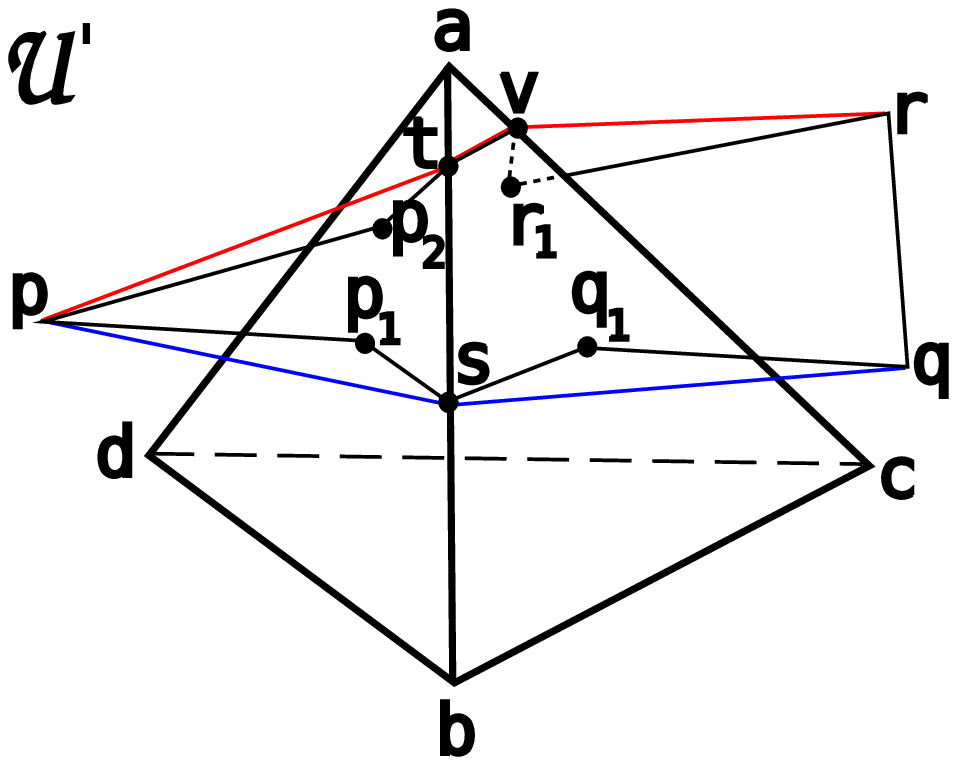}
         \caption{The geodesic triangle $\triangle(p,q,r)$ in $U'$ satisfies the CAT(0) inequality}
 \end{center}
\end{figure}

Let $\triangle (p'',s'',t'')$ be a comparison triangle in $\mathbf{R}^{2}$ for the geodesic
triangle $\triangle (p,s,t)$ in $U$. Let $\triangle (s'',t'',v'')$ be a comparison triangle in
$\mathbf{R}^{2}$ for the geodesic triangle $\triangle (s,t,v)$ in $U$.  We place the comparison
triangles $\triangle (p'',s'',t'')$ and $\triangle (s'',t'',v'')$ in different half-planes with
respect to the line $s''t''$ in $\mathbf{R}^{2}$. The CAT(0) inequality implies
$\angle_{p}(t,s) \leq \angle_{p''}(t'',s'').$

Let $\triangle (p''', v''',s''')$ be a comparison triangle in $\mathbf{R}^{2}$ for the geodesic
triangle $\triangle (p,v,s)$ in $U$. Let $\triangle (q''', s''', v''')$ be a comparison triangle
in $\mathbf{R}^{2}$ for the geodesic triangle $\triangle (q,s,v)$ in $U$.  We place the comparison
triangles $\triangle (p''', v''',s''')$ and $\triangle (q''', s''', v''')$ in different
half-planes with respect to the line $s'''v'''$ in $\mathbf{R}^{2}$. Let $t''' \in [p''',v''']$
be a comparison point for $t \in [p,v]$. Because $\angle_{t'''}(p''',s''') +
\angle_{t'''}(s''',v''') = \pi$, Alexandrov's Lemma ensures that
$\angle_{p''}(t'',s'') \leq \angle_{p'''}(t''',s''').$

Let $\triangle (p'^{V}, v'^{V},q'^{V})$ be a comparison triangle in $\mathbf{R}^{2}$ for the
geodesic triangle $\triangle (p,v,q)$ in $U$. Let $\triangle (v'^{V},q'^{V},r'^{V})$ be a
comparison triangle in $\mathbf{R}^{2}$ for the geodesic triangle $\triangle (v,q,r)$ in $U$.  We
place the comparison triangles $\triangle (p'^{V},v'^{V},q'^{V})$ and $\triangle
(v'^{V},q'^{V},r'^{V})$ in different half-planes with respect to the line $v'^{V}q'^{V}$ in
$\mathbf{R}^{2}$. Let $s'^{V} \in [p'^{V},q'^{V}]$ be a comparison point for $s \in [p,q]$.
Because $\angle_{s'^{V}}(p'^{V}, v'^{V}) + \angle_{s'^{V}}(v'^{V}, q'^{V}) = \pi$ and
$\angle_{v'}(p', q') + \angle_{v'}(q', r') = \pi$, Alexandrov's Lemma further
implies
$\angle_{p'''}(s''',v''') \leq \angle_{p'^{V}}(q'^{V},v'^{V})
\leq \angle_{p'}(q',r').$

Thus, $\angle_{p}(q,r) \leq \angle_{p'}(q',r').$ One can similarly show that
$\angle_{q}(r,p) \leq
\angle_{q'}(r',p')$ and $\angle_{r}(p,q) \leq
\angle_{r'}(p',q').$ So the geodesic triangle $\triangle(p,q,r)$ in $U'$ satisfies the
CAT(0) inequality.

\end{proof}

\begin{lemma}\label{3.33}
Let $r$ be a point in $U$ such that the geodesic segment $[r,q]$ does not intersect the
interior of $\sigma$
whereas the geodesic segment $[p,r]$ intersects the interior of $\tau_{1}$ in $p_{2}$, and
the interior of $\tau_{2}$ in
$r_{1}$. Let $s$ be a point on $e_{3}$ such that $\angle_{s}(a,q_{1}) =
\angle_{s}(b,c)$ and $\angle_{s}(a,b) = \angle_{s}(c,q_{1})$. Let
$t$ be a point on $e_{2}$ such that $\angle_{t}(a,p_{1}) =
\angle_{t}(c,d)$ and $\angle_{t}(a,c) = \angle_{t}(d,p_{1})$.  Let
$u$ be a point on $e_{2}$ such that $\angle_{u}(a,p_{2}) =
\angle_{u}(c,d)$ and $\angle_{u}(c,a) = \angle_{u}(d,p_{2})$. Let
$v$ be a point on $e_{3}$ such that $\angle_{v}(d,c) =
\angle_{v}(a,r_{1})$ and $\angle_{v}(d,a) = \angle_{v}(c,r_{1})$. If $d'(p,q) =
d'(p,t) + d'(t,s) + d'(s,q)$ and $d'(p,r) = d'(p,u) + d'(u,v) + d'(v,r)$, then the geodesic
triangle
$\triangle (p,q,r)$ in $U'$ satisfies the CAT(0) inequality.
\end{lemma}

\begin{proof}\label{3.34}

Because $U$ is a CAT(0) space, by Lemma \ref{3.7}, the points $s,t,u$ and $v$ exist and they are
unique.

\begin{figure}[h]
   \begin{center}
     \includegraphics[height=4cm]{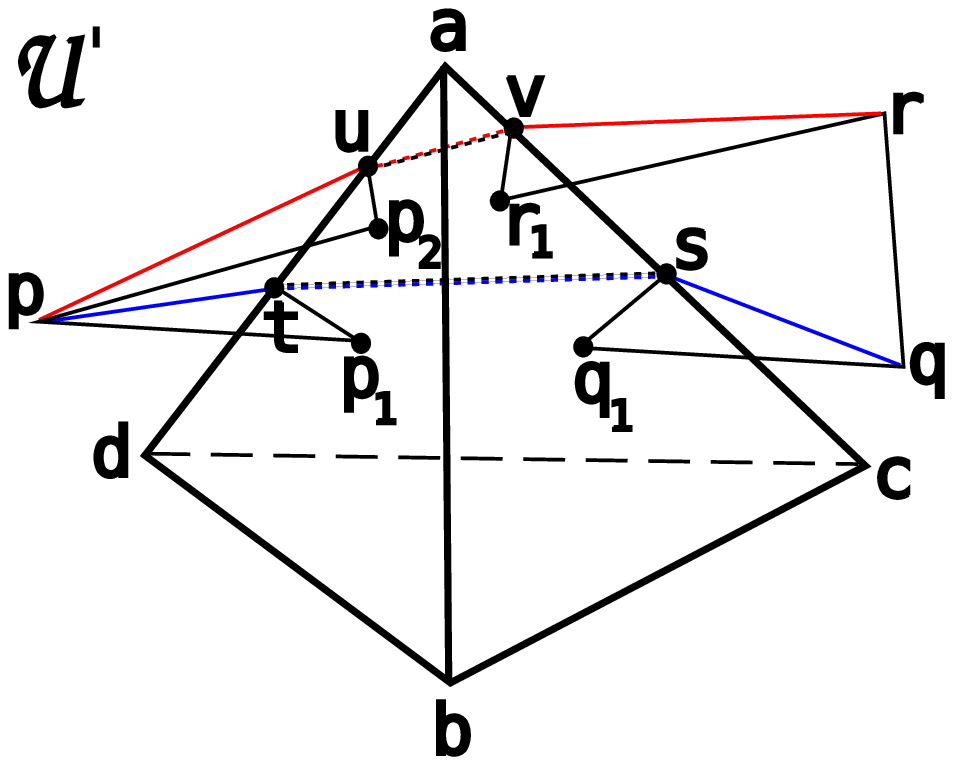}
        \caption{The geodesic triangle $\triangle(p,q,r)$ in $U'$ satisfies the CAT(0) inequality}
 \end{center}
\end{figure}

Let $\triangle (p',q',r')$ be a comparison triangle in $\mathbf{R}^{2}$ for the geodesic triangle
$\triangle (p,q,r)$ in $U'$.

Let $\triangle (p'',t'',u'')$ be a comparison triangle in $\mathbf{R}^{2}$ for the geodesic
triangle $\triangle (p,t,u)$ in $U$. Let $\triangle (t'',u'',v'')$ be a comparison triangle in
$\mathbf{R}^{2}$ for the geodesic triangle $\triangle (t,u,v)$ in $U$.  We place the comparison
triangles $\triangle (p'',t'',u'')$ and $\triangle (t'',u'',v'')$ in different half-planes with
respect to the line $t''u''$ in $\mathbf{R}^{2}$. By the CAT(0) inequality we have
$\angle_{p}(t,u) \leq \angle_{p''}(t'',u'').$

Let $\triangle (p''',t''',v''')$ be a comparison triangle in $\mathbf{R}^{2}$ for the geodesic
triangle $\triangle (p,t,v)$ in $U$. Let $\triangle (t''',v''',s''')$ be a comparison triangle
in $\mathbf{R}^{2}$ for the geodesic triangle $\triangle (t,v,s)$ in $U$.  We place the comparison
triangles $\triangle (p''',t''',v''')$ and $\triangle (t''',v''',s''')$ in different
half-planes with respect to the line $v'''t'''$ in $\mathbf{R}^{2}$. Alexandrov's Lemma implies that
$\angle_{p''}(t'',u'') \leq \angle_{p'''}(t''',v''').$

Let $\triangle (p'^{V}, s'^{V},v'^{V})$ be a comparison triangle in $\mathbf{R}^{2}$ for the
geodesic triangle $\triangle (p,s,v)$ in $U$. Let $\triangle (s'^{V},v'^{V}, q'^{V})$ be a
comparison triangle in $\mathbf{R}^{2}$ for the geodesic triangle $\triangle (s,v,q)$ in $U$.  We
place the comparison triangles $\triangle (p'^{V}, s'^{V},v'^{V})$ and $\triangle
(s'^{V},v'^{V}, q'^{V})$ in different half-planes with respect to the line $s'^{V}v'^{V}$ in
$\mathbf{R}^{2}$. Alexandrov's Lemma implies
$\angle_{p'''}(t''',v''') \leq \angle_{p'^{V}}(s'^{V},v'^{V}).$

Let $\triangle (p^{V}, q^{V},v^{V})$ be a comparison triangle in $\mathbf{R}^{2}$ for the
geodesic triangle $\triangle (p,q,v)$ in $U$. Let $\triangle (r^{V},q^{V},v^{V})$ be a
comparison triangle in $\mathbf{R}^{2}$ for the geodesic triangle $\triangle (r,q,v)$ in $U$.  We
place the comparison triangles $\triangle (p^{V}, q^{V},v^{V})$ and $\triangle
(r^{V},q^{V},v^{V})$ in different half-planes with respect to the line $q^{V}v^{V}$ in
$\mathbf{R}^{2}$. Alexandrov's Lemma ensures that
$\angle_{p'^{V}}(s'^{V},v'^{V}) \leq \angle_{p^{V}}(q^{V},v^{V}) \leq
\angle_{p'}(q',r').$

Altogether we have $\angle_{p}(q,r) \leq \angle_{p'}(q',r').$ One can show similarly that
$\angle_{q}(p,r) \leq \angle_{q'}(p',r')$ and
$\angle_{r}(p,q) \leq \angle_{r'}(p',q').$ So the geodesic triangle
$\triangle(p,q,r)$ in $U'$ satisfies the CAT(0) inequality.

\end{proof}

\begin{lemma}\label{3.35}
Let $r$ be a point in $U$ such that the geodesic segment $[r,q]$ intersects the interior
of $\tau_{3}$  in $r_{2}$, and the interior of $\tau_{2}$ in $q_{2}$ whereas the geodesic
segment $[p,r]$ intersects the interior of $\tau_{1}$ in $p_{2}$ and the interior of $\tau_{3}$
in $r_{1}$. Let $s$ be a point on $e_{1}$ such that $\angle_{s}(a,q_{1}) =
\angle_{s}(b,p_{1})$ and $\angle_{s}(a,p_{1}) = \angle_{s}(b,q_{1})$. Let
$t$ be a point on $e_{2}$ such that $\angle_{t}(a,p_{2}) =
\angle_{t}(d,r_{1})$ and $\angle_{t}(a,r_{1}) = \angle_{t}(d,p_{2})$. Let
$u$ be a point on $e_{3}$ such that $\angle_{u}(a,r_{2}) =
\angle_{u}(c,q_{2})$ and $\angle_{u}(a,q_{2}) = \angle_{u}(c,r_{2})$. If $d'(p,q) =
d'(p,s) + d'(s,q)$, $d'(p,r) = d'(p,t) + d'(t,r)$ and $d'(r,q) = d'(r,u) + d'(u,q)$, then the
geodesic triangle $\triangle (p,q,r)$ in $U'$ satisfies the CAT(0) inequality.
\end{lemma}

\begin{proof}\label{3.36}

Because $U$ is a CAT(0) space, Lemma \ref{3.7} implies that the points $s,t$ and $u$ exist and they
are unique.

\begin{figure}[h]
   \begin{center}
     \includegraphics[height=4cm]{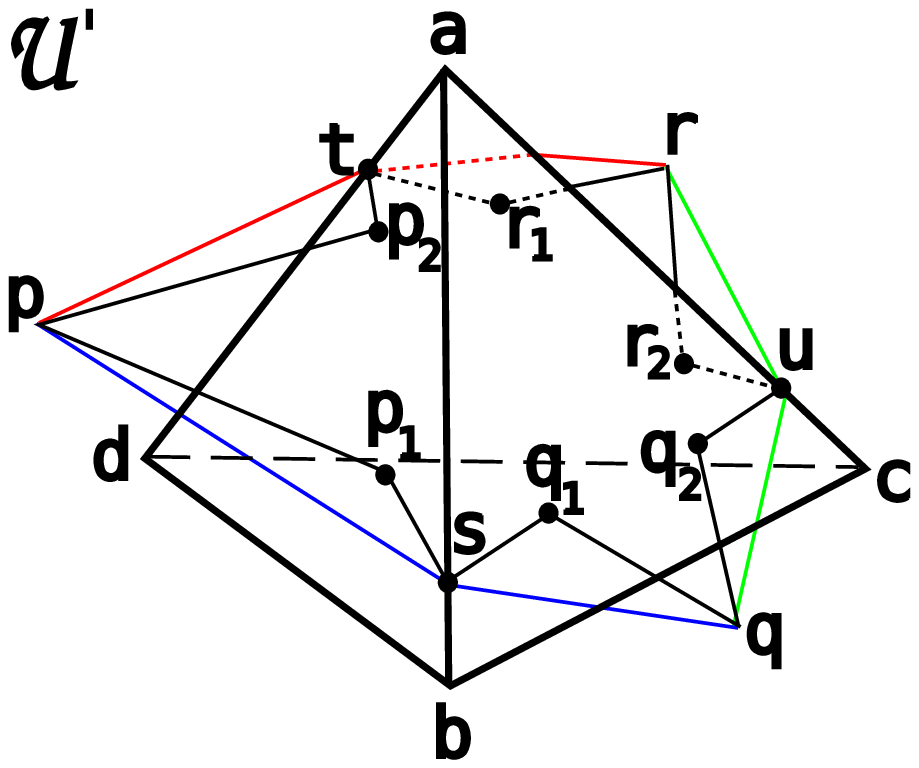}
         \caption{The geodesic triangle $\triangle(p,q,r)$ in $U'$ satisfies the CAT(0) inequality}
 \end{center}
\end{figure}

Let $\triangle (p',q',r')$ be a comparison triangle in $\mathbf{R}^{2}$ for the geodesic triangle
$\triangle (p,q,r)$ in $U'$. Let $s' \in [p',q']$ be a comparison point for $s \in [p,q]$.

Let $\triangle (p'',r'',s'')$ be a comparison triangle in $\mathbf{R}^{2}$ for the geodesic
triangle $\triangle (p,r,s)$ in $U$. Let $\triangle (r'',s'',q'')$ be a comparison triangle in
$\mathbf{R}^{2}$ for the geodesic triangle $\triangle (r,s,q)$ in $U$.  We place the comparison
triangles $\triangle (p'',r'',s'')$ and $\triangle (r'',s'',q'')$ in different half-planes with
respect to the line $r''s''$ in $\mathbf{R}^{2}$. The CAT(0) inequality implies that
$\angle_{p}(r,s) \leq \angle_{p''}(r'',s'')$.
Because $\angle_{s'}(p',r') + \angle_{s'}(r',q')
= \pi$, Alexandrov's Lemma implies $\angle_{p''}(r'',s'') \leq \angle_{p'}(r',s')$. So
$\angle_{p}(r,s) \leq \angle_{p'}(r',s')$.
One can similarly show that $\angle_{q}(p,r) \leq
\angle_{q'}(p',r')$ and $\angle_{r}(p,q) \leq \angle_{r'}(p',q')$.
Hence the
geodesic triangle $\triangle
(p,r,q)$ in $U'$ satisfies the CAT(0) inequality.

\end{proof}

The previous eight lemmas imply the following proposition.

\begin{proposition}\label{3.37}
The subcomplex $K'$ obtained by performing an elementary collapse on a finite, CAT(0) simplicial
$3$-complex $K$ satisfying Property A, is nonpositively curved.
\end{proposition}

\begin{proof}\label{3.38}
We must show that every point in $|K'|$ has a neighborhood which is a CAT(0) space.

Let $u,v,w$ be three distinct points in $U$ chosen such that they do
not belong to the interior of $\sigma$ and such that the geodesic segments $[u,v]$,
$[u,w]$ and $[v,w]$ in $U$ do not intersect the interior of $\sigma$. Note that the
geodesic triangle $\triangle (u,v,w)$ in $U'$ satisfies the CAT(0)
inequality. Hence, the Lemmas \ref{3.19}, \ref{3.21}, \ref{3.23}, \ref{3.25}, \ref{3.27},
\ref{3.29}, \ref{3.33} and
\ref{3.35} guarantee that any geodesic triangle in $U'$
fulfills the CAT(0) inequality. So $U'$ is a CAT(0) space.

Let $y$ be a point in $|K|$ that does not belong to the interior of $\sigma$. Let
$U_{y}$ be a neighborhood of $y$ homeomorphic to a closed ball of
radius $r_{y}$, $U_{y} = \{x \in |K| \mid d(y,x) \leq r_{y}\}$. The
radius $r_{y}$ is chosen small enough such that $U_{y}$ does not
intersect $\sigma$. For any $y$ in $|K'|$ that does not belong to
$\tau_{1}, \tau_{2}$ or $\tau_{3}$, we consider a neighborhood $U_{y}'$ that
coincides with $U_{y}$. $U'_{y}$ is thus a CAT(0) space.

So, because every point in $|K'|$ has a neighborhood which is a CAT(0) space, $|K'|$ is
nonpositively curved.

\end{proof}

The main result of the paper is an immediate consequence of the above proposition.

\begin{corollaryB}\label{3.39}
Any finite, CAT(0) simplicial $3$-complex $K$ that fulfills Property A, collapses to a point through CAT(0) subspaces.
\end{corollaryB}

\begin{proof}\label{3.40}

Because $K$ has a strongly convex metric, it has, by \cite{white_1967}, a $3$-simplex $\sigma$
with a free $2$-dimensional ($1$-dimensional) face. We fix a point $y$ in the interior
of a $3$-simplex of $K$. We define the mapping $R : |K| \times [0,1]
\rightarrow |K|$ which associates for any $x \in |K|$ and for any $t
\in [0,1]$, to $(x,t)$ the point a distance $t \cdot d(y,x)$ from
$x$ along the geodesic segment $[y,x]$. We note that $R$ is a continuous retraction of $|K|$ to
$y$. $R(|K| \times [0,1])$ is therefore contractible and then simply
connected. Let $m,n,s,t$ be the vertices of a tetrahedron $\delta$ in $R(|K| \times
[0,1])$ such that the segment $[m,n]$ either belongs to a
$1$-simplex ($2$-simplex) that is the face of a single $3$-simplex in the complex or it is itself
the $1$-dimensional face of a single $3$-simplex in the complex. For each tetrahedron $\delta$, we
deformation retract $R(|K| \times [0,1])$ by pushing in $\delta$ starting at $[m,n]$. We obtain
each time a subspace $|K'| = R(|K| \times [0,1])$ which remains
simply connected and, by Proposition \ref{3.37}, nonpositively
curved. So $|K'|$ is a CAT(0) space. Any two points in
$|K'|$ are therefore joined by a unique geodesic segment in $|K'|$. If at a
certain step we delete the point $y$, we fix another point in the
interior of a $3$-simplex of $K'$, define the mapping $R$ as before and
retract $|K'|$ by CAT(0) subspaces further. Because $K$ is finite
we reach, after a finite number of steps, a $2$-dimensional spine
$L$ which is also a CAT(0) space. So, by \cite{lazar_2010_8} (Theorem $3.1.10$), $L$ can be collapsed
further through CAT(0) subspaces to a point. Note that Property A refers only to segments intersecting $3$-simplices in $U$.
The $2$-dimensional spine $L$ does therefore no longer fulfill Property A. Hence we may indeed apply \cite{lazar_2010_8}, Theorem $3.1.10$.

\end{proof}


\begin{bibdiv}
\begin{biblist}



\bib{alex_1955}{article}{
   author={Aleksandrov, A.-D.},
   title={Die innere Geometrie der konvexen Flaechen},
   journal={Akademie Verlag, Berlin},
   date={1955},
}

\bib{aleksandrov_1986}{article}{
   author={Aleksandrov, A.-D.},
   author={Berestovski, V.-N.},
   author={Nikolaev, I.-G.},
   title={Generalized Riemannian spaces},
   journal={Russ. Math. Surveys},
   volume={41},
   date={1986},
   pages={3--44},
}



\bib{baralic_2014}{article}{
   author={Baralic, Dj.},
   author={Laz\u{a}r, I.-C.},
   title={A note on the combinatorial structure of finite and locally finite simplicial complexes of nonpositive curvature},
   journal={arxiv.org/pdf/$1403.4547$v$1$},
   date={2014},
}



\bib{benedetti_2012}{article}{
   author={Adiprasito, K. A.},
   author={Benedetti, B.},
   title={Metric Geometry, Convexity and Collapsibility},
   journal={arXiv:$1107.5789$v$4$},
   date={2013},
}

\bib{bridson_1999}{article}{
   author={Bridson, M.},
   author={Haefliger, A.},
   title={Metric spaces of non-positive curvature},
   journal={Springer, New York},
   date={1999},
}



\bib{burago_1992}{article}{
   author={Burago, Y.},
   author={Gromov, M.},
   author={Perelman, G.},
   title={Alexandrov spaces with curvature bounded below},
   journal={Russ. Math. Surveys},
   volume={47},
   date={1992},
   pages={1--58},
}

\bib{burago_2001}{article}{
   author={Burago, D.},
    author={Burago, Y.},
   author={Ivanov, S.},
   title={A Course in Metric Geometry},
   journal={American Mathematical Society, Providence, Rhode Island},
   date={2001},
}




\bib{chepoi_2009}{article}{
   author={Chepoi, V.},
   author={Osajda, D.},
   title={Dismantlability of weakly systolic complexes and applications},
   journal={Trans. Amer. Math. Soc.},
   number={367},
   volume={2},
   date={2015},
   pages={1247--1272},
}

\bib{crowley_2008}{article}{
   author={Crowley, K.},
   title={Discrete Morse theory and the geometry of nonpositively curved simplicial complexes},
   journal={Geometriae Dedicata},
   date={2008},
   pages={35--50},
}



\bib{davis_2008}{article}{
   author={Davis, M.},
   title={The geometry and topology of Coxeter groups},
   journal={London Mathematical Society Monographs Series, Princeton University Press, Princeton, NJ},
    volume={32},
   date={2008},
}




\bib{forman_1998}{article}{
   author={Forman, R.},
   title={Morse theory for cell complexes},
   journal={Adv. Math.},
   number={134},
   volume={1},
   date={1998},
   pages={90--145},
}



\bib{Gro}{article}{
   author={Gromov, M.},
   title={Hyperbolic groups},
   conference={
      title={Essays in group theory},
   },
   book={
      series={Math. Sci. Res. Inst. Publ.},
      volume={8},
      publisher={Springer, New York},
   },
   date={1987},
   pages={75--263},
}


\bib{JS1}{article}{
   author={Januszkiewicz, T.},
   author={{\'S}wi{\c{a}}tkowski, J.},
   title={Simplicial nonpositive curvature},
   journal={Publ. Math. Inst. Hautes \'Etudes Sci.},
   number={104},
   date={2006},
   pages={1--85},
   issn={0073-8301},
}


\bib{lazar_2009_2}{article}{
    title     ={CAT(0) simplicial complexes of dimension $2$ are collapsible},
    author    ={Laz\u{a}r, I.-C.},
    conference={
      title={Proceedings of the International Conference on Theory and Applications of Mathematics and
Informatics (Eds.: D. Breaz, N. Breaz, D. Wainberg), Alba Iulia, September, $9-11$, $2009$},
   },
   date={2009},
   pages={507--530},

}



\bib{lazar_2010_8}{article}{
    title     ={The study of simplicial complexes of nonpositive curvature},
    author    ={Laz\u{a}r, I.-C.},
    journal   ={Cluj University Press},
    date      ={2010},
    eprint    ={http://www.ioana-lazar.ro/phd.html}
}




\bib{white_1967}{article}{
   author={White, W.},
   title={A $2$-complex is collapsible if and only if it admits a strongly convex metric},
   journal={Fund. Math.},
   volume={68},
   date={1970},
   pages={23--29},
}

\end{biblist}
\end{bibdiv}
\end{document}